%% file: main.tex
\documentclass[11pt,a4paper]{article}
\usepackage[T1]{fontenc}
\usepackage[utf8]{inputenc}
\usepackage[margin=1in]{geometry}
\usepackage{amsmath,amssymb,amsfonts,amsthm,mathtools,mathrsfs}
\usepackage{enumitem,graphicx,booktabs}
\usepackage{placeins}
\usepackage[numbers,sort&compress]{natbib}
\usepackage{authblk}
\usepackage{microtype}
\usepackage[hidelinks]{hyperref}
\providecommand{\doi}[1]{\url{https://doi.org/#1}}
\hypersetup{
  pdftitle={Continuous-Time Wasserstein Distributionally Robust Control: Limits, Duality, and Stability},
  pdfauthor={Minjung Gim, Hee Jun Yang, and Yeoneung Kim}
}
\graphicspath{{./}}
\theoremstyle{plain}
\newtheorem{theorem}{Theorem}[section]
\newtheorem{lemma}[theorem]{Lemma}
\newtheorem{proposition}[theorem]{Proposition}
\newtheorem{corollary}[theorem]{Corollary}
\theoremstyle{definition}
\newtheorem{assumption}[theorem]{Assumption}
\newtheorem{example}[theorem]{Example}
\newtheorem{remark}[theorem]{Remark}
\newcommand{\R}{\mathbb{R}}
\newcommand{\N}{\mathbb{N}}

\newcommand{\dd}{\,\mathrm{d}}

\newcommand{\Acal}{\mathcal{A}}
\newcommand{\Bcal}{\mathcal{B}}
\newcommand{\Dcal}{\mathcal{D}}
\newcommand{\Pcal}{\mathcal{P}}

\newcommand{\Gammac}{\Gamma_{\mathrm c}}
\newcommand{\Gammaa}{\Gamma_{\mathrm a}}

\newcommand{\dH}{d_{\mathrm H}}
\newcommand{\BUC}{\mathrm{BUC}}

\setlist{topsep=3pt,itemsep=2pt}
\title{Continuous-Time Wasserstein Distributionally Robust Control:\\
Limits, Duality, and Stability}
\author[1]{Minjung Gim\thanks{Minjung Gim and Hee Jun Yang contributed equally to this work.}}
\author[1]{Hee Jun Yang\textsuperscript{\ensuremath{*}}}
\author[2]{Yeoneung Kim\thanks{Corresponding author: \texttt{yeoneung@yonsei.ac.kr}.}}
\affil[1]{National Institute for Mathematical Sciences, Daejeon 34047, Republic of Korea;
\texttt{minjunggim@gmail.com}, \texttt{yangheejun1009@gmail.com}}
\affil[2]{Department of Industrial Engineering, College of Engineering,
Yonsei University, Seoul 03722, Republic of Korea}
\date{}

\begin{document}
\maketitle

\begin{abstract}
A sampled robust control problem averages future values over disturbance
laws, while its continuous-time counterpart evolves with an averaged drift.
We establish a fluid limit connecting these two models for Wasserstein
ambiguity sets. Under a uniform disturbance moment bound and subcritical
growth, sampled Bellman values converge locally uniformly to the upper and
lower values of a discounted differential game. The argument accommodates
unbounded disturbances without requiring square-integrable increments.
We also derive an explicit bound on the change in either game value when
both the radius and center of the ambiguity set vary. For Euclidean
disturbance space, uniform Lipschitz continuity in the disturbance, common state
and cost Lipschitz bounds, and a discount rate exceeding the drift's state Lipschitz constant, this estimate converts
distributional confidence bounds into uniform value approximation bounds.
Static Wasserstein duality reformulates pointwise distributional optimization;
affine disturbance dependence permits an exact reduction to a ball of means.
Experiments illustrate the common limit and decreasing value error with
sample size. For a nonlinear disturbance model, uniform shifts overestimate
worst-case recovery speed by about 36\%. In a two-product inventory model,
increasing ambiguity raises production near balance while reducing the
region of full production in deeper backlogs.
\end{abstract}
\par

\section{Introduction}

Sampled distributionally robust control averages future values over the
disturbance law; the continuous-time model averages the drift over the disturbance law. These operations generally differ at a fixed step. We study
convergence as the step vanishes and sensitivity of the limiting values to
errors in the disturbance law.

With Wasserstein ambiguity \citep{YangCSL,YangDRO}, we scale state increments
with the time step and allow a new adversarial law at each time. This
rectangular specification \citep{Iyengar} leads to a discounted differential
game with measure-valued actions, averaged drift, and values determined by
the players' information pattern.

The paper has two main results. \emph{The sampled-to-relaxed limit}
(Theorem~\ref{thm:scheme-sampled}) establishes locally uniform convergence
for both orders of optimization. The growth conditions \(q_f,q_L<r\) allow
unbounded disturbances with a uniform \(r\)-moment bound. The proof uses
truncation to control large jumps and a common modulus of continuity to
obtain compactness; square-integrable increments are not required.
\emph{Joint radius and center stability} (Theorem~\ref{thm:sensitivity})
bounds the change in either limiting value by an explicit constant times
\(|\theta-\theta'|+W_r(\nu,\nu')\). The estimate holds on Euclidean
disturbance space under uniform Lipschitz continuity in the disturbance, common
state and cost Lipschitz bounds, and a discount rate exceeding the drift's state Lipschitz constant. It yields a
value approximation bound when the ambiguity radius is calibrated by a
distributional confidence bound (Corollary~\ref{cor:statistical}).

We identify the limits through the classical Elliott--Kalton dynamic programming principle (DPP) and
viscosity characterization \citep{ElliottKalton,ElliottKalton1974,
EvansSouganidis,BardiCapuzzo}. The measure-valued actions are relaxed
controls \citep{Warga,KrasovskiiSubbotin}; logarithmic localization handles
linear drift growth in the comparison argument. Sion's theorem
\citep{Sion} gives a sufficient Isaacs condition.
Static Wasserstein duality \citep{EsfahaniKuhn,BlanchetMurthy,GaoKleywegt}
then makes the Hamiltonian's distributional optimization accessible to
computation. For affine disturbance dependence it reduces exactly to a
ball of possible means. The value perturbation estimate complements the
static sensitivity analysis of Bartl et al.~\citep{BartlEtAl}.

Other continuous-time formulations use disturbance-energy penalties
\citep{BasarBernhard}, drift ambiguity \citep{ChenEpstein}, or volatility
ambiguity \citep{PengBook}. Jiang \citep{Jiang2026} studies causal DRO
with adapted Wasserstein penalties and a path-dependent Hamilton--Jacobi--Bellman
characterization. Our ambiguity sets constrain the law pointwise in time,
and fluid scaling produces a first-order equation on a finite-dimensional
state space. The convergence argument follows the monotonicity, stability,
and consistency approach of Barles and Souganidis~\citep{BarlesSouganidis}. A persistent diffusion
term requires a different scaling, as in stochastic differential games
\citep{FlemingSouganidis,FlemingSoner}.

Experiments show a common sampled and relaxed limit and decreasing value
errors with sample size. Nonlinear disturbance dependence makes transport
allocation consequential: uniform shifts overestimate worst-case recovery
speed by about 36\%. In the inventory model, increasing ambiguity raises
production near balance but reduces the full-production region in deeper
backlogs.

Sections~\ref{sec:setup}--\ref{sec:hji} establish the game and value equations;
Sections~\ref{sec:scheme}--\ref{sec:certificate} give the limit, duality, and
stability results; Section~\ref{sec:experiments} presents experiments.
All proofs are included with their results. Section~\ref{sec:discussion}
discusses modeling comparisons, and Appendix~\ref{app:experiments} supplies
additional experiments and numerical procedures.

\section{Model and Standing Assumptions}\label{sec:setup}

\subsection{Wasserstein Ambiguity Sets}\label{subsec:ambiguity}

Let the state space be \(\R^d\), the controller action space be
\(U\subset \R^m\), and the disturbance space be a closed set
\(W\subset \R^\ell\). Fix \(r\ge1\). Let \(\Pcal(W)\) denote the Borel
probability measures on \(W\) and
\(\Pcal_r(W):=\{\mu\in\Pcal(W):\int|w|^r\mu(\dd w)<\infty\}\). For
\(\mu,\nu\in\Pcal_r(W)\), the \(r\)-Wasserstein distance is
\begin{equation}\label{eq:wr}
  W_r(\mu,\nu)
  :=
  \left(
  \inf_{\kappa\in\Pi(\mu,\nu)}
  \int_{W\times W}|w-w'|^r\,\kappa(\dd w,\dd w')
  \right)^{1/r},
\end{equation}
where \(\Pi(\mu,\nu)\) denotes the set of transport plans with first marginal
\(\mu\) and second marginal \(\nu\); an optimal plan always exists
\citep[Thm.~4.1]{Villani}. Given a nominal law \(\nu\in\Pcal_r(W)\) and a radius \(\theta\ge0\), the
Wasserstein ambiguity set is
\begin{equation*}
  \Dcal_\theta(\nu)
  :=
  \{\mu\in\Pcal_r(W): W_r(\mu,\nu)\le \theta\}.
\end{equation*}
In the data-driven regime that motivates this paper, \(\nu\) is the empirical
distribution
\begin{equation}\label{eq:empirical}
  \nu_N:=\frac1N\sum_{i=1}^N\delta_{\widehat w_i}
\end{equation}
of \(N\) independent samples of an unknown true disturbance law \(\mu_*\),
and the radius \(\theta\) is calibrated by measure concentration so that
\(\mu_*\in\Dcal_\theta(\nu_N)\) with prescribed confidence
(Section~\ref{sec:certificate}). \subsection{Relaxed Dynamics}\label{subsec:relaxed}

The controlled dynamics are specified by a vector field
\(f:\R^d\times U\times W\to\R^d\). For a probability measure \(\mu\) on
\(W\), define the \emph{relaxed drift}
\begin{equation}\label{eq:relaxed-drift}
  F(x,u,\mu):=\int_W f(x,u,w)\,\mu(\dd w)
\end{equation}
whenever the integral converges. At time \(s\ge0\), the adversary selects
a law from \(\Dcal\subset\Pcal_r(W)\), usually a Wasserstein ball
\(\Dcal_\theta(\nu)\), and the state evolves with the corresponding drift
\eqref{eq:relaxed-drift}. Linearity in \(\mu\) and convexity of the ball
will allow us to apply Sion's theorem in Proposition~\ref{prop:isaacs}.

The admissible controls and adversarial paths are
\begin{align*}
  \Acal&:=\{a:[0,\infty)\to U: a \text{ is Lebesgue measurable}\},\\
  \Bcal&:=\{b:[0,\infty)\to\Dcal: b \text{ is Lebesgue measurable}\}.
\end{align*}
Here, measurability is understood with respect to the Lebesgue
\(\sigma\)-algebra on \([0,\infty)\) and the Borel \(\sigma\)-algebras on
\(U\) and \(\Dcal\), respectively, where \(\Dcal\) carries the topology of
weak convergence of probability measures. Controls and adversarial paths
are identified up to Lebesgue-a.e.
equivalence. For \(a\in\Acal\), \(b\in\Bcal\), and \(x\in\R^d\), the relaxed
state equation is
\begin{equation}\label{eq:state-ode}
  \dot X_s = F(X_s,a(s),b(s)),
  \qquad X_0=x,
\end{equation}
or, equivalently, in integral form,
\begin{equation*}
  X_t=x+\int_0^tF(X_s,a(s),b(s))\,\dd s.
\end{equation*}
Well-posedness of \eqref{eq:state-ode} under the standing assumptions is
proved in Section~\ref{sec:wellposed}; the unique solution is denoted by
\(X^{x,a,b}\).

The pointwise constraint on \(b\) is rectangular: paths can be concatenated
without exhausting a cumulative budget. Section~\ref{sec:discussion}, compares this specification with budget
constraints and examines how relaxation changes the information pattern.

\subsection{Cost, Information Pattern, and Elliott--Kalton Values}
\label{subsec:values}

For \(a\in\Acal\) and \(b\in\Bcal\), the infinite-horizon discounted cost is
\begin{equation}\label{eq:cost}
  J(x;a,b):=\int_0^\infty e^{-\rho s}c(X_s^{x,a,b},a(s))\,\dd s,
\end{equation}
where the running cost \(c:\R^d\times U\to[0,M_c]\) and the discount rate
\(\rho>0\) are as in Assumption~\ref{ass:standing} below. By boundedness of
\(c\),
\begin{equation}\label{eq:cost-bound}
  0\le J(x;a,b)\le \frac{M_c}{\rho}.
\end{equation}
The controller minimizes cost and the adversary maximizes it. We specify
the information pattern through Elliott--Kalton nonanticipative strategies
\citep{EvansSouganidis}.

A controller nonanticipative strategy is a map \(\alpha:\Bcal\to\Acal\) such
that, for every \(b_1,b_2\in\Bcal\) and \(t\ge0\),
\begin{equation*}
  b_1=b_2\text{ a.e. on }[0,t]
  \quad\Longrightarrow\quad
  \alpha[b_1]=\alpha[b_2]\text{ a.e. on }[0,t].
\end{equation*}

Here \(b_1=b_2\) a.e. means that there is one Lebesgue-null time set outside
which \(b_1(s)\) and \(b_2(s)\) agree as Borel probability measures on \(W\).
Equivalently, their distance in any metric for the weak topology is zero
a.e.; the diagonal in \(\Dcal\times\Dcal\) is Borel. Thus the condition is
measurable and independent of the representatives of the paths.
Let \(\Gammac\) denote the collection of such strategies. Similarly, an
adversary nonanticipative strategy is a map \(\beta:\Acal\to\Bcal\) such that
\begin{equation*}
  a_1=a_2\text{ a.e. on }[0,t]
  \quad\Longrightarrow\quad
  \beta[a_1]=\beta[a_2]\text{ a.e. on }[0,t],
\end{equation*}
and the set of such strategies is denoted by \(\Gammaa\). Like controls and
adversarial paths, strategies are identified up to Lebesgue-a.e. equivalence
of their outputs.

The lower and upper Elliott--Kalton values are
\begin{align*}
  V^-(x)&:=\inf_{\alpha\in\Gammac}\sup_{b\in\Bcal}J(x;\alpha[b],b),
  \\
  V^+(x)&:=\sup_{\beta\in\Gammaa}\inf_{a\in\Acal}J(x;a,\beta[a]).
\end{align*}
Both values lie in \([0,M_c/\rho]\). In \(V^-\), the controller responds to
the adversary's path; in \(V^+\), the roles reverse. The corresponding
Hamiltonian orders are explained in Remark~\ref{rem:information}.

\begin{remark}[Ordering and Information]\label{rem:no-ordering}
The values optimize over different strategy--path pairs, so the elementary
minimax inequality does not order them. Their ordering follows from the
Hamilton--Jacobi comparison principle (Proposition~\ref{prop:ordering}).
\end{remark}

\subsection{Standing Assumptions}\label{subsec:assumptions}

The following assumptions are maintained unless explicitly stated otherwise.

\begin{assumption}[Standing hypotheses]\label{ass:standing}
The following conditions hold.
\begin{enumerate}[label=(A\arabic*)]
  \item \(U\subset \R^m\) is compact and nonempty. The ambiguity set
  \(\Dcal\subset\Pcal_r(W)\) is nonempty, compact and metrizable for the
  topology of weak convergence of probability measures, and has uniformly bounded
  \(r\)-th moments,
  \(m_r(\Dcal):=\sup_{\mu\in\Dcal}\int_W|w|^r\,\mu(\dd w)<\infty\).
  \item The relaxed drift
  \(F:\R^d\times U\times\Dcal\to\R^d\) is continuous, where \(\Dcal\) carries
  the weak topology of (A1). There exist constants \(L_F,C_F>0\) such that, for
  all \(x,y\in\R^d\), \(u\in U\), and \(\mu\in\Dcal\),
  \begin{equation*}
    |F(x,u,\mu)-F(y,u,\mu)|\le L_F|x-y|,
    \qquad
    |F(x,u,\mu)|\le C_F(1+|x|).
  \end{equation*}
  \item The running cost \(c:\R^d\times U\to[0,M_c]\) is bounded and
  continuous on \(\R^d\times U\), and uniformly continuous in \(x\)
  uniformly in \(u\in U\); \(\omega_c\) denotes a (bounded, nondecreasing)
  modulus of continuity of \(c\) in the state variable, uniform in
  \(u\in U\). \item The discount rate satisfies \(\rho>0\).
\end{enumerate}
\end{assumption}

\subsection{Wasserstein Balls and Induced Drifts}
\label{sec:setting}

The next lemmas verify (A1) for Wasserstein balls and (A2) under growth
conditions on \(f\).

\begin{lemma}\label{lem:ball-compact}
Let \(W\subset\R^\ell\) be closed, \(\nu\in\Pcal_r(W)\), and \(\theta\ge0\).
Then \(\Dcal_\theta(\nu)\) is nonempty, convex, and compact and metrizable for
the topology of weak convergence, with
\begin{equation}\label{eq:uniform-moment}
  m_r(\Dcal_\theta(\nu))^{1/r}
  \le \theta+\Bigl(\int_W|w|^r\,\nu(\dd w)\Bigr)^{1/r}.
\end{equation}
Moreover, on \(\Dcal_\theta(\nu)\) the weak topology coincides with the
\(W_q\)-topology for every \(q\in[1,r)\).
\end{lemma}

\begin{proof}
\emph{Convexity.} Let \(\mu_0,\mu_1\in\Dcal_\theta(\nu)\), \(t\in[0,1]\), and
let \(\kappa_i\in\Pi(\mu_i,\nu)\) be optimal plans. Then
\(t\kappa_0+(1-t)\kappa_1\in\Pi(t\mu_0+(1-t)\mu_1,\nu)\) and the transport cost
is linear in the plan, whence
\(W_r^r(t\mu_0+(1-t)\mu_1,\nu)\le tW_r^r(\mu_0,\nu)+(1-t)W_r^r(\mu_1,\nu)
\le\theta^r\).

\emph{Moment bound.} For \(\mu\in\Dcal_\theta(\nu)\) and an optimal
\(\kappa\in\Pi(\mu,\nu)\), the triangle inequality in \(L^r(\kappa)\) gives
\((\int|w|^r\dd\mu)^{1/r}\le(\int|w-w'|^r\dd\kappa)^{1/r}
+(\int|w'|^r\dd\nu)^{1/r}\), which is \eqref{eq:uniform-moment}.

\emph{Compactness.} By \eqref{eq:uniform-moment} and Markov's inequality,
\(\mu(\{|w|>K\})\le m_r/K^r\) uniformly on \(\Dcal_\theta(\nu)\), so the set is
tight. Prokhorov's theorem gives relative compactness in the weak topology,
which is metrizable because \(W\) is Polish. If \(\mu_n\in\Dcal_\theta(\nu)\) and
\(\mu_n\rightharpoonup\mu\) weakly, then \(\mu\) is supported in the closed set
\(W\), \(\int|w|^r\dd\mu\le\liminf_n\int|w|^r\dd\mu_n\le m_r\) by the
portmanteau theorem, and
\(W_r(\mu,\nu)\le\liminf_nW_r(\mu_n,\nu)\le\theta\) by the lower semicontinuity
of \(W_r\) with respect to weak convergence \citep[Rem.~6.12]{Villani}. Hence
\(\Dcal_\theta(\nu)\) is weakly closed, and therefore compact.

\emph{Topologies.} For \(q\in[1,r)\), \(W_q\)-convergence is equivalent to weak
convergence plus convergence of \(q\)-th moments \citep[Thm.~6.9]{Villani}. On
\(\Dcal_\theta(\nu)\), the uniform \(r\)-moment bound makes \(|w|^q\) uniformly
integrable, so weak convergence already implies convergence of \(q\)-th
moments. The two topologies therefore coincide on the ball.
\end{proof}

\begin{lemma}\label{lem:F-regular}
Let \(\Dcal\) satisfy (A1), let \(m:=m_r(\Dcal)\), and let
\(f:\R^d\times U\times W\to\R^d\) be continuous with, for some
\(q_f,q_L\in[0,r)\),
\begin{align}
  |f(x,u,w)-f(y,u,w)|&\le L_f(1+|w|^{q_L})|x-y|,\label{eq:f-lip}\\
  |f(x,u,w)|&\le C_f(1+|x|)(1+|w|^{q_f}).\label{eq:f-growth}
\end{align}
Then \(F(x,u,\mu):=\int_Wf(x,u,w)\,\mu(\dd w)\) satisfies (A2), with
\(L_F=L_f(1+m^{q_L/r})\) and \(C_F=C_f(1+m^{q_f/r})\).
\end{lemma}

\begin{proof}
The Lipschitz and growth bounds follow from \eqref{eq:f-lip},
\eqref{eq:f-growth}, and Jensen's inequality,
\(\int|w|^{q}\dd\mu\le(\int|w|^r\dd\mu)^{q/r}\le m^{q/r}\) for \(q\le r\).

To prove joint continuity, take \((x_n,u_n,\mu_n)\to(x,u,\mu)\) in
\(\R^d\times U\times\Dcal\), with \(|x_n|\le R\) for all \(n\), and split
\[
\begin{aligned}
  |F(x_n,u_n,\mu_n)-F(x,u,\mu)|
  &\le\int_W|f(x_n,u_n,w)-f(x,u,w)|\,\mu_n(\dd w)\\
  &\quad+\Bigl|\int_Wf(x,u,w)\,(\mu_n-\mu)(\dd w)\Bigr|.
\end{aligned}
\]
For the first term, fix \(K>0\) and split at \(|w|\le K\). On
\(\{|w|\le K\}\), \(f\) is uniformly continuous on the compact set
\(\bar B_R\times U\times(W\cap\bar B_K)\), so the integrand converges to \(0\)
uniformly. On the tail, \eqref{eq:f-growth} and Markov's inequality give, uniformly in
\(n\),
\[
  \int_{\{|w|>K\}}\!2C_f(1+R)(1+|w|^{q_f})\,\mu_n(\dd w)
  \le 2C_f(1+R)\Bigl(\frac{m}{K^r}+\frac{m}{K^{r-q_f}}\Bigr)
  \xrightarrow[K\to\infty]{}0,
\]
since \(q_f<r\). For the second term, weak convergence applies to bounded
continuous truncations of \(f(x,u,\cdot)\). The same tail estimate allows
the truncation to be removed, so this term also tends to zero.
\end{proof}

\begin{remark}[Strictness of \(q_f<r\)]\label{rem:strict}
The strict inequality \(q_f<r\) in Lemma~\ref{lem:F-regular} cannot in general
be relaxed to \(q_f=r\) when \(W\) is unbounded: mass of size \(1/n\) placed at
distance \(n\theta\) stays in the \(W_1\)-ball but destroys weak continuity of
\(\mu\mapsto\int f\dd\mu\) for \(f\) of linear growth. Consequently, for the
important affine case \(f(x,u,w)=f_0(x,u)+G(x,u)w\) (so \(q_f=1\)), one should
take \(r>1\), work with a compact disturbance set \(W\), or use the
finite-dimensional reduction of Corollary~\ref{cor:mean-set}, which is valid
for every \(r\ge1\).
\end{remark}

Extensions to disturbance-dependent costs are described in Remark~\ref{rem:cost-mu}.

\subsection{State Equations}\label{sec:wellposed}

For measurable actions, the Carath\'eodory ODE \eqref{eq:state-ode} has the
following flow estimates for dynamic programming and comparison.

\begin{proposition}\label{prop:ode}
Under Assumption~\ref{ass:standing}, for every \(x\in\R^d\),
\(a\in\Acal\), and \(b\in\Bcal\), equation \eqref{eq:state-ode} admits a unique
global absolutely continuous solution \(X^{x,a,b}\). Moreover, for all
\(t\ge0\),
\begin{equation}\label{eq:state-growth}
  |X_t^{x,a,b}|\le (1+|x|)e^{C_Ft}-1,
\end{equation}
and, for all \(x,y\in\R^d\),
\begin{equation}\label{eq:state-lip-initial}
  |X_t^{x,a,b}-X_t^{y,a,b}|
  \le e^{L_Ft}|x-y|.
\end{equation}
\end{proposition}

\begin{proof}
For fixed \(z\), the map \(s\mapsto F(z,a(s),b(s))\) is measurable by
composition. It is globally Lipschitz in \(z\) and has linear growth, so
the Carath\'eodory theorem gives a unique maximal solution. The estimate
\[
  |X_t|\le |x|+\int_0^t C_F(1+|X_s|)\,\dd s
\]
and Gronwall's inequality give \eqref{eq:state-growth}, which rules out
finite-time blow-up. Applying Gronwall once more to
\[
  |X_t^{x,a,b}-X_t^{y,a,b}|
  \le |x-y|+L_F\int_0^t|X_s^{x,a,b}-X_s^{y,a,b}|\,\dd s
\]
gives \eqref{eq:state-lip-initial}.
\end{proof}

For later use we record a short-time estimate. Fix \(R>0\). By
\eqref{eq:state-growth}, there is a constant \(C_R>0\) such that, for all
\(|x|\le R\), all \(a\in\Acal\), all \(b\in\Bcal\), and all \(t\in[0,1]\),
\begin{equation}\label{eq:short-time}
  \sup_{0\le s\le t}|X_s^{x,a,b}-x|\le C_Rt.
\end{equation}

\section{Dynamic Programming and Hamilton--Jacobi Equations}\label{sec:ek-dpp}\label{sec:hji}

Rectangularity allows the game to restart from its current state. Continuity
supports dynamic programming, which gives the value equations; comparison
then establishes uniqueness and orders the values.

\begin{lemma}\label{lem:value-continuity}
There exists a modulus of continuity \(\omega_J\), independent of \(a\) and
\(b\), such that
\begin{equation}\label{eq:J-uniform-continuity}
  |J(x;a,b)-J(y;a,b)|\le \omega_J(|x-y|)
\end{equation}
for all \(x,y\in\R^d\), \(a\in\Acal\), and \(b\in\Bcal\). In particular,
\begin{equation*}
  |V^\pm(x)-V^\pm(y)|\le \omega_J(|x-y|).
\end{equation*}
\end{lemma}

\begin{proof}
By \eqref{eq:state-lip-initial},
\(|X_s^{x,a,b}-X_s^{y,a,b}|\le e^{L_Fs}|x-y|\). Hence
\[
  |J(x;a,b)-J(y;a,b)|
  \le \int_0^\infty e^{-\rho s}
  \omega_c(e^{L_Fs}|x-y|)\,\dd s.
\]
Set
\begin{equation}\label{eq:omega-J}
  \omega_J(r):=\int_0^\infty e^{-\rho s}\omega_c(e^{L_Fs}r)\,\dd s.
\end{equation}
Since \(\omega_c\) is bounded, dominated convergence gives
\(\omega_J(r)\to0\) as \(r\downarrow0\). This proves
\eqref{eq:J-uniform-continuity}. The estimates for \(V^-\) and \(V^+\) follow
by taking the corresponding suprema and infima.
\end{proof}

For \(T>0\), write
\begin{equation*}
  I_T(x;a,b):=\int_0^T e^{-\rho s}c(X_s^{x,a,b},a(s))\,\dd s.
\end{equation*}

For \(T>0\), define time shifts
\begin{equation*}
  a^T(s):=a(T+s),\qquad b^T(s):=b(T+s),\qquad s\ge0,
\end{equation*}
and concatenations
\begin{equation*}
  (a\otimes_T\widetilde a)(s):=
  \begin{cases}
    a(s),&0\le s<T,\\
    \widetilde a(s-T),&s\ge T,
  \end{cases}
\end{equation*}
with \(b\otimes_T\widetilde b\) defined analogously. The sets \(\Acal\) and
\(\Bcal\) are stable under time shifts and concatenations.

\begin{lemma}\label{lem:flow-cost}
For every \(T>0\), \(t\ge0\), \(x\in\R^d\), \(a\in\Acal\), and \(b\in\Bcal\),
\begin{equation}\label{eq:flow-property}
  X_{T+t}^{x,a,b}=X_t^{X_T^{x,a,b},a^T,b^T}.
\end{equation}
Consequently,
\begin{equation}\label{eq:cost-splitting}
  J(x;a,b)=
  \int_0^T e^{-\rho s}c(X_s^{x,a,b},a(s))\,\dd s
  +e^{-\rho T}J(X_T^{x,a,b};a^T,b^T).
\end{equation}
\end{lemma}

\begin{proof}
The map \(t\mapsto X_{T+t}^{x,a,b}\) and the map
\(t\mapsto X_t^{X_T^{x,a,b},a^T,b^T}\) solve the same initial value problem on
\([0,\infty)\). Uniqueness gives \eqref{eq:flow-property}. Splitting the
integral in \eqref{eq:cost} at \(T\), changing variables \(s=T+r\), and using
\eqref{eq:flow-property} gives \eqref{eq:cost-splitting}.
\end{proof}

The DPP follows the classical strategy-concatenation argument of
Elliott and Kalton~\citep[Thm.~3.1]{ElliottKalton1974} and
Evans and Souganidis~\citep[Thm.~3.1]{EvansSouganidis}. Elliott and Kalton allow compact metric
action spaces, so the maximizing and minimizing spaces can be taken as
\(\Dcal\) and \(U\). Evans and Souganidis treat finite horizons with
compact Euclidean action spaces and bounded Lipschitz data. Here we retain the
concatenation construction and use the state estimate and common cost
modulus to handle linear drift growth and uniformly continuous costs.
For the corresponding horizon-\(S\) values \(V_S^\pm\) with zero terminal
payoff, the discounted tail gives
\[
  \|V^\pm-V_S^\pm\|_\infty
  \le \frac{M_c}{\rho}e^{-\rho S}.
\]
Thus the discounted tail vanishes for every \(\rho>0\). We give the direct
infinite-horizon argument, using Lemma~\ref{lem:value-continuity} to choose
continuation strategies on a finite cover of the reachable states.

\begin{theorem}\label{thm:dpp}
Under Assumption~\ref{ass:standing}, for every \(T>0\) and \(x\in\R^d\),
\begin{align}
V^-(x)
&=
\inf_{\alpha\in\Gammac}\sup_{b\in\Bcal}
\left\{
  I_T(x;\alpha[b],b)+e^{-\rho T}V^-(X_T^{x,\alpha[b],b})
\right\},
\label{eq:dpp-minus}\\
V^+(x)
&=
\sup_{\beta\in\Gammaa}\inf_{a\in\Acal}
\left\{
  I_T(x;a,\beta[a])+e^{-\rho T}V^+(X_T^{x,a,\beta[a]})
\right\}.
\label{eq:dpp-plus}
\end{align}
\end{theorem}

\begin{proof}

\smallskip
\noindent\textbf{1. The lower value.}
Write \(W_T^-(x)\) for the right-hand side of \eqref{eq:dpp-minus}.
To show \(V^-\ge W_T^-\), fix \(\alpha\in\Gammac\) and \(\eta>0\). For a prefix \(b^0\in\Bcal\),
define the continuation strategy
\begin{equation*}
  \alpha_{T,b^0}[\widetilde b](s):=
  \alpha[b^0\otimes_T\widetilde b](T+s),\qquad s\ge0.
\end{equation*}
If \(\widetilde b_1=\widetilde b_2\) a.e. on \([0,t]\), then
\(b^0\otimes_T\widetilde b_1=b^0\otimes_T\widetilde b_2\) a.e. on
\([0,T+t]\). By nonanticipativity of \(\alpha\),
\(\alpha_{T,b^0}[\widetilde b_1]=\alpha_{T,b^0}[\widetilde b_2]\) a.e. on
\([0,t]\). Thus \(\alpha_{T,b^0}\in\Gammac\).

Let \(y=X_T^{x,\alpha[b^0],b^0}\). By the definition of \(V^-\),
\[
  V^-(y)
  \le \sup_{\widetilde b\in\Bcal}
  J(y;\alpha_{T,b^0}[\widetilde b],\widetilde b).
\]
Choose \(\widetilde b\in\Bcal\) such that
\[
  J(y;\alpha_{T,b^0}[\widetilde b],\widetilde b)
  \ge V^-(y)-\eta.
\]
Set \(b=b^0\otimes_T\widetilde b\). Then \(b=b^0\) a.e. on \([0,T]\), so
\(\alpha[b]=\alpha[b^0]\) a.e. on \([0,T]\), and the state at time \(T\) is
\(y\). Lemma~\ref{lem:flow-cost} gives
\[
\begin{aligned}
  J(x;\alpha[b],b)
  &=I_T(x;\alpha[b^0],b^0)
  +e^{-\rho T}J(y;\alpha_{T,b^0}[\widetilde b],\widetilde b)\\
  &\ge I_T(x;\alpha[b^0],b^0)+e^{-\rho T}V^-(y)-e^{-\rho T}\eta.
\end{aligned}
\]
Taking the supremum over \(b^0\), then the infimum over \(\alpha\), and finally
letting \(\eta\downarrow0\), yields
\begin{equation}\label{eq:dpp-minus-ge}
  V^-(x)\ge W_T^-(x).
\end{equation}

We prove the reverse inequality. Let \(\eta>0\), and choose
\(\bar\alpha\in\Gammac\) such that
\begin{equation}\label{eq:alpha-bar-eps}
  \sup_{b\in\Bcal}
  \{I_T(x;\bar\alpha[b],b)+e^{-\rho T}V^-(X_T^{x,\bar\alpha[b],b})\}
  \le W_T^-(x)+\eta.
\end{equation}
By \eqref{eq:state-growth}, all states reachable from \(x\) up to time \(T\)
belong to the compact ball
\begin{equation*}
  K_T(x):=\{z\in\R^d: |z|\le (1+|x|)e^{C_FT}-1\}.
\end{equation*}
Let \(\delta>0\), choose a finite \(\delta\)-net
\(\{y_1,\ldots,y_N\}\subset K_T(x)\), and define the Borel index map
\(\iota:K_T(x)\to\{1,\ldots,N\}\) by
\(\iota(y):=\min\{i:|y-y_i|=\min_j|y-y_j|\}\), so that
\(|y-y_{\iota(y)}|\le\delta\) for all \(y\in K_T(x)\). For each \(i\), choose
\(\alpha^i\in\Gammac\) satisfying
\begin{equation}\label{eq:alpha-i-eps}
  \sup_{\widetilde b\in\Bcal}J(y_i;\alpha^i[\widetilde b],\widetilde b)
  \le V^-(y_i)+\eta.
\end{equation}
Define \(\alpha^*\) by
\begin{equation*}
  \alpha^*[b](s):=
  \begin{cases}
    \bar\alpha[b](s),&0\le s<T,\\
    \alpha^{\iota(y_b)}[b^T](s-T),&s\ge T,
  \end{cases}
  \qquad y_b:=X_T^{x,\bar\alpha[b],b}.
\end{equation*}
The output of \(\alpha^*\) is measurable by concatenation.
If \(b_1=b_2\) a.e. on \([0,t]\) with \(t\le T\), equality of the outputs
follows from nonanticipativity of \(\bar\alpha\). For \(t>T\), the paths
reach the same state \(y_{b_1}=y_{b_2}\) at \(T\), hence select the same
continuation strategy, whose nonanticipativity gives equality on \([T,t]\).
Therefore \(\alpha^*\in\Gammac\).

For any \(b\in\Bcal\), by Lemma~\ref{lem:flow-cost},
\eqref{eq:alpha-i-eps}, and Lemma~\ref{lem:value-continuity},
\[
\begin{aligned}
  J(x;\alpha^*[b],b)
  &=I_T(x;\bar\alpha[b],b)
  +e^{-\rho T}J(y_b;\alpha^{\iota(y_b)}[b^T],b^T)\\
  &\le I_T(x;\bar\alpha[b],b)
  +e^{-\rho T}\{V^-(y_{\iota(y_b)})+\eta+
     \omega_J(|y_b-y_{\iota(y_b)}|)\}\\
  &\le I_T(x;\bar\alpha[b],b)
  +e^{-\rho T}\{V^-(y_b)+\eta+2\omega_J(\delta)\}.
\end{aligned}
\]
By definition, \(V^-(x)\le \sup_b J(x;\alpha^*[b],b)\). Taking the
supremum in the preceding estimate and using \eqref{eq:alpha-bar-eps} yields
\[
  V^-(x)\le W_T^-(x)+\eta+e^{-\rho T}\{\eta+2\omega_J(\delta)\}.
\]
Letting \(\delta\downarrow0\) and \(\eta\downarrow0\) gives
\begin{equation}\label{eq:dpp-minus-le}
  V^-(x)\le W_T^-(x).
\end{equation}
Combining \eqref{eq:dpp-minus-ge} and \eqref{eq:dpp-minus-le} proves
\eqref{eq:dpp-minus}.

\smallskip
\noindent\textbf{2. The upper value.}
Write \(W_T^+(x)\) for the right-hand side of \eqref{eq:dpp-plus}.
For \(\eta>0\), fix \(\beta\in\Gammaa\), \(a^0\in\Acal\), and define
\begin{equation*}
  \beta_{T,a^0}[\widetilde a](s):=
  \beta[a^0\otimes_T\widetilde a](T+s),\qquad s\ge0.
\end{equation*}
As above, \(\beta_{T,a^0}\in\Gammaa\). With
\(y=X_T^{x,a^0,\beta[a^0]}\), choose \(\widetilde a\in\Acal\) such that
\[
  J(y;\widetilde a,\beta_{T,a^0}[\widetilde a])\le V^+(y)+\eta.
\]
Setting \(a=a^0\otimes_T\widetilde a\) and using Lemma~\ref{lem:flow-cost}, we
obtain
\[
  J(x;a,\beta[a])
  \le I_T(x;a^0,\beta[a^0])+e^{-\rho T}V^+(y)+e^{-\rho T}\eta.
\]
Taking the infimum over \(a^0\), then the supremum over \(\beta\), and
letting \(\eta\downarrow0\) gives
\begin{equation}\label{eq:dpp-plus-le}
  V^+(x)\le W_T^+(x).
\end{equation}

For the reverse inequality, fix \(\eta>0\) and choose \(\bar\beta\in\Gammaa\)
such that
\begin{equation}\label{eq:beta-bar-eps}
  \inf_{a\in\Acal}
  \{I_T(x;a,\bar\beta[a])+e^{-\rho T}V^+(X_T^{x,a,\bar\beta[a]})\}
  \ge W_T^+(x)-\eta.
\end{equation}
On the same \(\delta\)-net \(\{y_1,\ldots,y_N\}\) of \(K_T(x)\),
choose \(\beta^i\in\Gammaa\) for each \(i\) such that
\begin{equation*}
  \inf_{\widetilde a\in\Acal}J(y_i;\widetilde a,\beta^i[\widetilde a])
  \ge V^+(y_i)-\eta.
\end{equation*}
Define
\begin{equation*}
  \beta^*[a](s):=
  \begin{cases}
    \bar\beta[a](s),&0\le s<T,\\
    \beta^{\iota(y_a)}[a^T](s-T),&s\ge T,
  \end{cases}
  \qquad y_a:=X_T^{x,a,\bar\beta[a]}.
\end{equation*}
The same concatenation argument shows that \(\beta^*\in\Gammaa\). For every
\(a\in\Acal\),
\[
\begin{aligned}
  J(x;a,\beta^*[a])
  &=I_T(x;a,\bar\beta[a])
  +e^{-\rho T}J(y_a;a^T,\beta^{\iota(y_a)}[a^T])\\
  &\ge I_T(x;a,\bar\beta[a])
  +e^{-\rho T}\{V^+(y_{\iota(y_a)})-\eta
     -\omega_J(|y_a-y_{\iota(y_a)}|)\}\\
  &\ge I_T(x;a,\bar\beta[a])
  +e^{-\rho T}\{V^+(y_a)-\eta-2\omega_J(\delta)\}.
\end{aligned}
\]
Taking the infimum over \(a\) and using \eqref{eq:beta-bar-eps}
gives the following bound, since
\(V^+(x)\ge \inf_a J(x;a,\beta^*[a])\):
\[
  V^+(x)\ge W_T^+(x)-\eta-e^{-\rho T}\{\eta+2\omega_J(\delta)\}.
\]
Letting \(\delta\downarrow0\) and \(\eta\downarrow0\) gives
\begin{equation*}
  V^+(x)\ge W_T^+(x).
\end{equation*}
Together with \eqref{eq:dpp-plus-le}, this proves \eqref{eq:dpp-plus}.
\end{proof}

\subsection{Hamiltonians and Viscosity Solutions}

Define, for \((x,p)\in\R^d\times\R^d\),
\begin{align}
  H^+(x,p)&:=\sup_{\mu\in\Dcal}\inf_{u\in U}
  \{c(x,u)+p\cdot F(x,u,\mu)\},
  \label{eq:Hplus}\\
  H^-(x,p)&:=\inf_{u\in U}\sup_{\mu\in\Dcal}
  \{c(x,u)+p\cdot F(x,u,\mu)\}.
  \label{eq:Hminus}
\end{align}
The notation follows the Elliott--Kalton convention: \(V^-\) solves the
Hamilton--Jacobi equation with \(H^+\), and \(V^+\) solves the one with
\(H^-\). The minimax inequality gives
\begin{equation}\label{eq:H-minimax}
  H^+(x,p)\le H^-(x,p).
\end{equation}

\begin{remark}[Information Pattern]\label{rem:information}
In \(V^-\), the minimizing player observes the current adversarial action,
so its minimization is inside \(H^+=\sup_\mu\inf_u\). The roles are
reversed for \(V^+\) and \(H^-\). We use no-delay Elliott--Kalton
strategies throughout; information restrictions for implemented controllers
are discussed in Remark~\ref{rem:certificate-info}.
\end{remark}

\begin{lemma}\label{lem:H-cont}
Under Assumption~\ref{ass:standing}, the functions \(H^+\) and \(H^-\) are
continuous on \(\R^d\times\R^d\).
\end{lemma}

\begin{proof}
The function
\[
  L(x,p,u,\mu):=c(x,u)+p\cdot F(x,u,\mu)
\]
is continuous on \(\R^d\times\R^d\times U\times\Dcal\). Since
\(U\times\Dcal\) is compact, Berge's maximum theorem implies that the partial
infima and suprema appearing in \eqref{eq:Hplus} and \eqref{eq:Hminus} are
continuous.
\end{proof}

The Hamiltonians also satisfy the following structure estimates, which are the
key to the comparison principle.

\begin{lemma}\label{lem:H-structure}
Let \(H\in\{H^+,H^-\}\). Under Assumption~\ref{ass:standing}, for all
\(x,y,p,q\in\R^d\),
\begin{align}
  |H(x,p)-H(y,p)|&\le \omega_c(|x-y|)+L_F|p|\,|x-y|,
  \label{eq:H-x-module}\\
  |H(x,p)-H(x,q)|&\le C_F(1+|x|)\,|p-q|.
  \label{eq:H-p-module}
\end{align}
In particular, with \(\omega_1(t):=\omega_c(t)+L_Ft\),
\begin{equation}\label{eq:H2-module}
  |H(x,p)-H(y,p)|\le \omega_1\bigl((1+|p|)|x-y|\bigr).
\end{equation}
\end{lemma}

\begin{proof}
For all \((u,\mu)\in U\times\Dcal\),
\(|L(x,p,u,\mu)-L(y,p,u,\mu)|\le\omega_c(|x-y|)+L_F|p||x-y|\) by (A2)--(A3),
and \(|L(x,p,u,\mu)-L(x,q,u,\mu)|\le C_F(1+|x|)|p-q|\). Infima and suprema of a
family of functions sharing a common modulus inherit that modulus, which gives
\eqref{eq:H-x-module}--\eqref{eq:H-p-module}. Estimate \eqref{eq:H2-module}
follows because \(\omega_c\) is nondecreasing and \((1+|p|)\ge1\).
\end{proof}

Write \(\BUC(\R^d)\) for the bounded uniformly continuous real-valued
functions on \(\R^d\). A function \(v\in\BUC(\R^d)\) is a viscosity subsolution of
\begin{equation*}
  \rho v(x)-H(x,Dv(x))=0
\end{equation*}
if, whenever \(\varphi\in C^1(\R^d)\) and \(v-\varphi\) has a local maximum at
\(x_0\),
\begin{equation*}
  \rho\varphi(x_0)-H(x_0,D\varphi(x_0))\le0,
\end{equation*}
where, as usual, \(\varphi\) is normalized so that \(\varphi(x_0)=v(x_0)\).
It is a viscosity supersolution if, whenever \(v-\varphi\) has a local minimum
at \(x_0\) (with the same normalization),
\begin{equation*}
  \rho\varphi(x_0)-H(x_0,D\varphi(x_0))\ge0.
\end{equation*}
It is a viscosity solution if it is both a subsolution and a supersolution.

The viscosity inequalities use a common short-time expansion. Let
\(\varphi\in C^1(\R^d)\), \(x_0\in\R^d\), and define
\begin{equation*}
  \mathscr L_\varphi(x,u,\mu):=
  c(x,u)+D\varphi(x)\cdot F(x,u,\mu)-\rho\varphi(x).
\end{equation*}
For each fixed \(R>0\), uniformly for \(|x_0|\le R\), \(a\in\Acal\), and
\(b\in\Bcal\),
\begin{equation}\label{eq:short-consistency}
\begin{aligned}
&\int_0^h e^{-\rho s}c(X_s^{x_0,a,b},a(s))\,\dd s
+e^{-\rho h}\varphi(X_h^{x_0,a,b})-\varphi(x_0)\\
&\qquad =
\int_0^h e^{-\rho s}\mathscr L_\varphi(x_0,a(s),b(s))\,\dd s
+h\,o_h(1),
\end{aligned}
\end{equation}
where \(o_h(1)\to0\) as \(h\downarrow0\), uniformly over the initial states
and action paths specified above.

Indeed, the chain rule gives the same identity with
\(\mathscr L_\varphi(X_s^{x_0,a,b},a(s),b(s))\). The replacement of
\(X_s^{x_0,a,b}\) by \(x_0\) follows from \eqref{eq:short-time} and the uniform
continuity of \(\mathscr L_\varphi\) on compact state balls times
\(U\times\Dcal\).

We use the derivation of viscosity inequalities from the DPP in
Evans and Souganidis~\citep[Thm.~4.1]{EvansSouganidis}, using \eqref{eq:short-consistency} and
measurable near-optimizers on the compact action spaces.

\begin{theorem}\label{thm:hji}
Under Assumption~\ref{ass:standing}, the lower value \(V^-\) is a bounded
uniformly continuous viscosity solution of
\begin{equation}\label{eq:HJ-lower}
  \rho V^-(x)-H^+(x,DV^-(x))=0,
  \qquad x\in\R^d,
\end{equation}
and the upper value \(V^+\) is a bounded uniformly continuous viscosity
solution of
\begin{equation}\label{eq:HJ-upper}
  \rho V^+(x)-H^-(x,DV^+(x))=0,
  \qquad x\in\R^d.
\end{equation}
\end{theorem}

\begin{proof}
The cost bound \eqref{eq:cost-bound} and Lemma~\ref{lem:value-continuity}
give boundedness and uniform continuity. It remains to prove the viscosity inequalities. Fix \(u_\star\in U\) and
\(\mu_\star\in\Dcal\) to extend the short-time action paths by constants.

\smallskip
\noindent\textbf{1. The equation for \(V^-\).}
Let \(\varphi\in C^1(\R^d)\) and assume that \(V^- -\varphi\) has a local
maximum at \(x_0\). By subtracting a constant from \(\varphi\), we may assume
\(V^-(x_0)=\varphi(x_0)\) and \(V^-\le\varphi\) in a neighborhood of \(x_0\).
Let \(p=D\varphi(x_0)\).

Fix \(\varepsilon>0\). The action spaces are compact metric and the payoff
\((u,\mu)\mapsto c(x_0,u)+p\cdot F(x_0,u,\mu)\) is continuous.
Measurable selection \citep[Thm.~18.19]{AliprantisBorder} gives a Borel map
\(u_\varepsilon:\Dcal\to U\) such that
\begin{equation}\label{eq:u-eps-selector}
  c(x_0,u_\varepsilon(\mu))+p\cdot F(x_0,u_\varepsilon(\mu),\mu)
  \le
  \inf_{u\in U}\{c(x_0,u)+p\cdot F(x_0,u,\mu)\}+\varepsilon
\end{equation}
for all \(\mu\in\Dcal\). Define a controller strategy by
\(\alpha_\varepsilon[b](s):=u_\varepsilon(b(s))\) for \(s\in[0,h]\) and
\(\alpha_\varepsilon[b](s):=u_\star\) for \(s>h\). Composition with the Borel
selector preserves measurability, and the pointwise response is
nonanticipative. The DPP gives
\[
\begin{aligned}
  V^-(x_0)
  &\le \sup_{b\in\Bcal}
  \left\{
    \int_0^h e^{-\rho s}c(X_s^{x_0,\alpha_\varepsilon[b],b},
      \alpha_\varepsilon[b](s))\,\dd s
    +e^{-\rho h}V^-(X_h^{x_0,\alpha_\varepsilon[b],b})
  \right\}.
\end{aligned}
\]
For small \(h\), all trajectories remain in the neighborhood where
\(V^-\le\varphi\). Hence
\[
  0\le \sup_{b\in\Bcal}
  \left\{
    \int_0^h e^{-\rho s}c(X_s,\alpha_\varepsilon[b](s))\,\dd s
    +e^{-\rho h}\varphi(X_h)-\varphi(x_0)
  \right\}.
\]
Using \eqref{eq:short-consistency}, dividing by \(h\), and sending
\(h\downarrow0\), we obtain
\[
  0\le
  \sup_{\mu\in\Dcal}
  \{c(x_0,u_\varepsilon(\mu))+p\cdot F(x_0,u_\varepsilon(\mu),\mu)
  -\rho\varphi(x_0)\}.
\]
By \eqref{eq:u-eps-selector},
\[
  0\le H^+(x_0,p)+\varepsilon-\rho\varphi(x_0).
\]
Letting \(\varepsilon\downarrow0\) gives
\[
  \rho\varphi(x_0)-H^+(x_0,D\varphi(x_0))\le0.
\]
Thus \(V^-\) is a subsolution of \eqref{eq:HJ-lower}.

Now assume that \(V^- -\varphi\) has a local minimum at \(x_0\), with
\(V^-(x_0)=\varphi(x_0)\) and \(V^-\ge\varphi\) near \(x_0\), and set
\(p:=D\varphi(x_0)\). For each small
\(h>0\), choose an \(h^2\)-optimal strategy \(\alpha_h\in\Gammac\) in the DPP:
\begin{equation}\label{eq:alpha-h-opt}
\begin{aligned}
  &\sup_{b\in\Bcal}\bigg\{
    \int_0^h e^{-\rho s}c(X_s^{x_0,\alpha_h[b],b},\alpha_h[b](s))\,\dd s
    +e^{-\rho h}V^-(X_h^{x_0,\alpha_h[b],b})\bigg\}\\
  &\qquad\le V^-(x_0)+h^2.
\end{aligned}
\end{equation}
Fix \(\mu\in\Dcal\), and let \(b^\mu(s)\equiv\mu\). Since \(V^-\ge\varphi\)
near \(x_0\), \eqref{eq:alpha-h-opt} implies
\[
  \int_0^h e^{-\rho s}c(X_s,\alpha_h[b^\mu](s))\,\dd s
  +e^{-\rho h}\varphi(X_h)-\varphi(x_0)
  \le h^2.
\]
Using \eqref{eq:short-consistency}, dividing by \(h\), and sending
\(h\downarrow0\), we get
\[
  0\ge
  \inf_{u\in U}\{c(x_0,u)+p\cdot F(x_0,u,\mu)-\rho\varphi(x_0)\}
\]
for every \(\mu\in\Dcal\). Taking the supremum over \(\mu\) gives
\[
  0\ge H^+(x_0,p)-\rho\varphi(x_0),
\]
that is,
\[
  \rho\varphi(x_0)-H^+(x_0,D\varphi(x_0))\ge0.
\]
Therefore \(V^-\) is a supersolution of \eqref{eq:HJ-lower}.

\smallskip
\noindent\textbf{2. The equation for \(V^+\).}
Let \(\varphi\in C^1(\R^d)\), and suppose that \(V^+ -\varphi\) has a local
maximum at \(x_0\), with equality at \(x_0\) and \(V^+\le\varphi\) near
\(x_0\); set \(p:=D\varphi(x_0)\). For each small \(h>0\), choose an
\(h^2\)-optimal adversary strategy
\(\beta_h\in\Gammaa\) in the DPP:
\begin{equation}\label{eq:beta-h-opt-sub}
\begin{aligned}
  \inf_{a\in\Acal}
  \left\{
    \int_0^h e^{-\rho s}c(X_s^{x_0,a,\beta_h[a]},a(s))\,\dd s
    +e^{-\rho h}V^+(X_h^{x_0,a,\beta_h[a]})
  \right\}
  \ge V^+(x_0)-h^2.
\end{aligned}
\end{equation}
Fix \(u\in U\), and set \(a^u(s)\equiv u\). Since \(V^+\le\varphi\) near
\(x_0\), \eqref{eq:beta-h-opt-sub} gives
\[
  \varphi(x_0)-h^2
  \le
  \int_0^h e^{-\rho s}c(X_s,a^u(s))\,\dd s
  +e^{-\rho h}\varphi(X_h),
\]
where \(X_s=X_s^{x_0,a^u,\beta_h[a^u]}\). By
\eqref{eq:short-consistency},
\[
  0\le
  \frac1h\int_0^h e^{-\rho s}
  \mathscr L_\varphi(x_0,u,\beta_h[a^u](s))\,\dd s+o_h(1).
\]
Since
\[
  \mathscr L_\varphi(x_0,u,\beta_h[a^u](s))
  \le \sup_{\mu\in\Dcal}\{c(x_0,u)+p\cdot F(x_0,u,\mu)-\rho\varphi(x_0)\},
\]
we get
\[
  0\le \sup_{\mu\in\Dcal}
  \{c(x_0,u)+p\cdot F(x_0,u,\mu)-\rho\varphi(x_0)\}
\]
for every \(u\in U\). Taking the infimum over \(u\) yields
\[
  \rho\varphi(x_0)-H^-(x_0,D\varphi(x_0))\le0.
\]
Thus \(V^+\) is a subsolution of \eqref{eq:HJ-upper}.

Finally suppose that \(V^+ -\varphi\) has a local minimum at \(x_0\), with
\(V^+(x_0)=\varphi(x_0)\) and \(V^+\ge\varphi\) near \(x_0\), and set
\(p:=D\varphi(x_0)\). Fix
\(\varepsilon>0\). By measurable selection
\citep[Thm.~18.19]{AliprantisBorder}, there exists a Borel map
\(\mu_\varepsilon:U\to\Dcal\) such that, for all \(u\in U\),
\begin{equation}\label{eq:mu-eps-selector}
  c(x_0,u)+p\cdot F(x_0,u,\mu_\varepsilon(u))
  \ge
  \sup_{\mu\in\Dcal}\{c(x_0,u)+p\cdot F(x_0,u,\mu)\}-\varepsilon.
\end{equation}
Define \(\beta_\varepsilon[a](s):=\mu_\varepsilon(a(s))\) for \(s\in[0,h]\) and
\(\beta_\varepsilon[a](s):=\mu_\star\) for \(s>h\). This is a nonanticipative
adversary strategy, and \(s\mapsto\mu_\varepsilon(a(s))\) is measurable as
before. The DPP and \(V^+\ge\varphi\) give
\[
\begin{aligned}
  \varphi(x_0)=V^+(x_0)&\ge\inf_{a\in\Acal}\bigg\{
    \int_0^h e^{-\rho s}c(X_s^{x_0,a,\beta_\varepsilon[a]},a(s))\,\dd s\\
    &\hspace{15mm}+e^{-\rho h}\varphi(X_h^{x_0,a,\beta_\varepsilon[a]})\bigg\}.
\end{aligned}
\]
Using \eqref{eq:short-consistency}, dividing by \(h\), and letting
\(h\downarrow0\), we obtain
\[
  0\ge
  \inf_{u\in U}\{c(x_0,u)+p\cdot F(x_0,u,\mu_\varepsilon(u))
  -\rho\varphi(x_0)\}.
\]
By \eqref{eq:mu-eps-selector},
\[
  0\ge H^-(x_0,p)-\varepsilon-\rho\varphi(x_0).
\]
Letting \(\varepsilon\downarrow0\) gives
\[
  \rho\varphi(x_0)-H^-(x_0,D\varphi(x_0))\ge0.
\]
This proves the supersolution property for \(V^+\) and completes the four
viscosity inequalities.
\end{proof}

\begin{remark}[Measurable selectors]
The selectors used in the proof can be avoided by finite \(\varepsilon\)-nets
in \(U\) and \(\Dcal\). For each \(\varepsilon>0\), compactness gives finitely many actions,
and a Borel partition selects an action satisfying the required inequality.
This construction applies to compact metric action spaces.
\end{remark}

\subsection{Comparison, Uniqueness, and the Ordering of the Values}

\begin{theorem}\label{thm:comparison}
Let Assumption~\ref{ass:standing} hold and let \(H\in\{H^+,H^-\}\). If
\(v_1\in\BUC(\R^d)\) is a viscosity subsolution and \(v_2\in\BUC(\R^d)\) is a
viscosity supersolution of \(\rho v-H(x,Dv)=0\) on \(\R^d\), then
\(v_1\le v_2\) on \(\R^d\). In particular, \eqref{eq:HJ-lower} and
\eqref{eq:HJ-upper} each admit at most one solution in \(\BUC(\R^d)\), namely
\(V^-\) and \(V^+\) respectively.
\end{theorem}

\begin{proof}
We use the classical comparison structure estimates
\citep{BardiCapuzzo,Ishii1984,CIL}, doubling variables and localizing
with a logarithmic penalty. Fix \(H\in\{H^+,H^-\}\), and let
\(v_1,v_2\in\BUC(\R^d)\) be a viscosity subsolution and supersolution,
respectively, of
\(\rho v-H(x,Dv)=0\) on \(\R^d\); \(\omega_2\) denotes a (bounded,
nondecreasing) modulus of continuity of \(v_2\), and
\(K:=\|v_1\|_\infty+\|v_2\|_\infty\). The proof uses continuity of \(H\)
and the structure estimates \eqref{eq:H2-module} and \eqref{eq:H-p-module}.

The localization function is
\begin{equation*}
  g(x):=\log\bigl(1+|x|^2\bigr),
\end{equation*}
which satisfies \(g\in C^\infty(\R^d)\), \(g\ge0\), \(g(0)=0\),
\(g(x)\to\infty\) as \(|x|\to\infty\), and
\begin{equation}\label{eq:g-gradient}
  |Dg(x)|=\frac{2|x|}{1+|x|^2}\le1,
  \qquad
  (1+|x|)\,|Dg(x)|=\frac{2|x|+2|x|^2}{1+|x|^2}\le3,
  \qquad x\in\R^d.
\end{equation}
By \eqref{eq:H-p-module}, the gradient perturbation contributes at most
\(C_F(1+|x|)\,\beta|Dg(x)|\). The second bound in \eqref{eq:g-gradient}
therefore makes this error uniformly of order \(\beta\).

Suppose, for contradiction, that
\(M:=\sup_{\R^d}(v_1-v_2)>0\). For \(\varepsilon\in(0,1]\) and
\(\beta\in(0,1]\), define
\begin{equation}\label{eq:doubling}
  \Phi_{\varepsilon,\beta}(x,y)
  :=v_1(x)-v_2(y)-\frac{|x-y|^2}{2\varepsilon}
  -\beta\bigl(g(x)+g(y)\bigr),
  \qquad (x,y)\in\R^d\times\R^d.
\end{equation}
Boundedness of \(v_1,v_2\) and coercivity of \(g\) imply that the continuous
function \(\Phi_{\varepsilon,\beta}\) attains a maximum at
\((\bar x,\bar y)\). Set
\(M_{\varepsilon,\beta}:=\Phi_{\varepsilon,\beta}(\bar x,\bar y)\) and
\[
  M_\beta:=\sup_{x\in\R^d}\Phi_{\varepsilon,\beta}(x,x)
  =\sup_{x\in\R^d}\bigl[v_1(x)-v_2(x)-2\beta g(x)\bigr]
  \le M_{\varepsilon,\beta}.
\]

\smallskip
\noindent\textbf{Step 1 (\(M_\beta\to M\)).}
Clearly \(M_\beta\le M\) since \(g\ge0\). For \(k\in\N\), pick \(x_k\) with
\((v_1-v_2)(x_k)\ge M-1/k\); then \(M_\beta\ge M-1/k-2\beta g(x_k)\ge M-2/k\)
for all \(\beta\) small enough. Hence \(M_\beta\uparrow M\) as
\(\beta\downarrow0\).

\smallskip
\noindent\textbf{Step 2 (penalization estimates).}
From \(\Phi_{\varepsilon,\beta}(\bar x,\bar y)\ge
\Phi_{\varepsilon,\beta}(0,0)=v_1(0)-v_2(0)\ge-K\) and \(g(0)=0\),
\[
  \frac{|\bar x-\bar y|^2}{2\varepsilon}
  +\beta\bigl(g(\bar x)+g(\bar y)\bigr)
  \le v_1(\bar x)-v_2(\bar y)+K\le2K,
\]
so that, uniformly in \(\beta\in(0,1]\),
\begin{equation}\label{eq:xy-close}
  |\bar x-\bar y|\le r_\varepsilon:=2\sqrt{K\varepsilon}.
\end{equation}
Next, from \(\Phi_{\varepsilon,\beta}(\bar x,\bar y)\ge
\Phi_{\varepsilon,\beta}(\bar x,\bar x)\), the fact that \(g\) is
\(1\)-Lipschitz (by \eqref{eq:g-gradient}), and \(\beta\le1\),
\begin{equation}\label{eq:refined}
  \frac{|\bar x-\bar y|^2}{2\varepsilon}
  \le v_2(\bar x)-v_2(\bar y)+\beta\bigl(g(\bar x)-g(\bar y)\bigr)
  \le \omega_2(|\bar x-\bar y|)+|\bar x-\bar y|
  \le \omega_2(r_\varepsilon)+r_\varepsilon.
\end{equation}
Writing \(p_\varepsilon:=(\bar x-\bar y)/\varepsilon\), estimates
\eqref{eq:xy-close}--\eqref{eq:refined} give
\begin{equation}\label{eq:sigma-def}
  (1+|p_\varepsilon|)\,|\bar x-\bar y|
  \le|\bar x-\bar y|+\frac{|\bar x-\bar y|^2}{\varepsilon}
  \le 3r_\varepsilon+2\omega_2(r_\varepsilon)=:\sigma(\varepsilon)
  \xrightarrow[\varepsilon\downarrow0]{}0,
\end{equation}
uniformly in \(\beta\in(0,1]\).

\smallskip
\noindent\textbf{Step 3 (viscosity inequalities).}
The map \(x\mapsto\Phi_{\varepsilon,\beta}(x,\bar y)\) attains a global
maximum at \(\bar x\); equivalently, \(v_1-\varphi_1\) attains a maximum at
\(\bar x\) with the \(C^1\) test function
\(\varphi_1(x):=|x-\bar y|^2/(2\varepsilon)+\beta g(x)+\mathrm{const}\), whose
gradient is \(D\varphi_1(\bar x)=p_\varepsilon+\beta Dg(\bar x)\). The
subsolution property gives
\[
  \rho v_1(\bar x)\le H\bigl(\bar x,\,p_\varepsilon+\beta Dg(\bar x)\bigr).
\]
Similarly, \(y\mapsto\Phi_{\varepsilon,\beta}(\bar x,y)\) attains a maximum at
\(\bar y\), i.e., \(v_2-\psi\) attains a \emph{minimum} at \(\bar y\) with
\(\psi(y):=-|\bar x-y|^2/(2\varepsilon)-\beta g(y)+\mathrm{const}\) and
\(D\psi(\bar y)=p_\varepsilon-\beta Dg(\bar y)\); the supersolution property
gives
\[
  \rho v_2(\bar y)\ge H\bigl(\bar y,\,p_\varepsilon-\beta Dg(\bar y)\bigr).
\]
Subtracting and inserting intermediate terms,
\[
\begin{aligned}
  \rho\bigl(v_1(\bar x)-v_2(\bar y)\bigr)
  &\le
  \bigl[H(\bar x,p_\varepsilon+\beta Dg(\bar x))-H(\bar x,p_\varepsilon)\bigr]
  +\bigl[H(\bar x,p_\varepsilon)-H(\bar y,p_\varepsilon)\bigr]\\
  &\qquad
  +\bigl[H(\bar y,p_\varepsilon)-H(\bar y,p_\varepsilon-\beta Dg(\bar y))\bigr]\\
  &\le
  C_F(1+|\bar x|)\,\beta|Dg(\bar x)|
  +C_F(1+|\bar y|)\,\beta|Dg(\bar y)|\\
  &\qquad+\omega_1\bigl((1+|p_\varepsilon|)|\bar x-\bar y|\bigr)\\
  &\le
  6C_F\beta+\omega_1\bigl(\sigma(\varepsilon)\bigr),
\end{aligned}
\]
where we used \eqref{eq:H-p-module} for the first and third brackets, \eqref{eq:H2-module} for the second, and
\eqref{eq:g-gradient}, \eqref{eq:sigma-def} in the last line.

\smallskip
\noindent\textbf{Step 4 (conclusion).}
Since the penalization terms in \eqref{eq:doubling} are nonnegative,
\[
  v_1(\bar x)-v_2(\bar y)\ \ge\ \Phi_{\varepsilon,\beta}(\bar x,\bar y)
  =M_{\varepsilon,\beta}\ \ge\ M_\beta,
\]
and therefore
\[
  \rho M_\beta\ \le\ 6C_F\beta+\omega_1\bigl(\sigma(\varepsilon)\bigr).
\]
Letting \(\varepsilon\downarrow0\) with \(\beta\) fixed (using
\eqref{eq:sigma-def} and \(\omega_1(0+)=0\)) yields
\(\rho M_\beta\le6C_F\beta\); letting \(\beta\downarrow0\) and using Step 1
yields \(\rho M\le0\), contradicting \(M>0\) and \(\rho>0\). Hence
\(v_1\le v_2\) on \(\R^d\).
\end{proof}

\begin{remark}
Comparison holds for every \(\rho>0\), independently of \(L_F\) and
\(C_F\). Only the supersolution requires uniform continuity, through
\(\omega_2\) in \eqref{eq:refined}. For the subsolution, boundedness and
continuity suffice for \eqref{eq:xy-close} and attainment of the maximum
in \eqref{eq:doubling}.
\end{remark}

\begin{proposition}\label{prop:ordering}
Under Assumption~\ref{ass:standing},
\begin{equation*}
  V^-(x)\le V^+(x),\qquad x\in\R^d.
\end{equation*}
\end{proposition}

\begin{proof}
By \eqref{eq:H-minimax} and Theorem~\ref{thm:hji}, \(V^-\) is a
subsolution of the equation for \(V^+\); Theorem~\ref{thm:comparison}
gives the ordering; see also Evans and Souganidis~\citep[Cor.~4.2(i)]{EvansSouganidis}.
\end{proof}

\subsection{A Verifiable Isaacs Condition}

We say that the \emph{Isaacs condition} holds if
\begin{equation}\label{eq:isaacs}
  H^+(x,p)=H^-(x,p)=:H(x,p)
  \qquad\text{for all }(x,p)\in\R^d\times\R^d.
\end{equation}
The conditions below permit an application of Sion's minimax theorem
\citep[Thm.~3.4]{Sion}: the Wasserstein ball is convex
(Lemma~\ref{lem:ball-compact}) and the relaxed drift
\eqref{eq:relaxed-drift} is linear in the measure.

\begin{proposition}\label{prop:isaacs}
Let Assumption~\ref{ass:standing} hold, let \(\Dcal\) be convex (e.g.,
\(\Dcal=\Dcal_\theta(\nu)\)), and let \(F\) be induced by \(f\) as in
\eqref{eq:relaxed-drift}, so that \(\mu\mapsto F(x,u,\mu)\) is affine.
Suppose in addition that \(U\) is convex and that, for every
\((x,p,\mu)\in\R^d\times\R^d\times\Dcal\), the map
\begin{equation*}
  u\;\longmapsto\; c(x,u)+p\cdot F(x,u,\mu)
\end{equation*}
is quasiconvex on \(U\). (A sufficient condition: \(c(x,\cdot)\) is convex and
\(f(x,\cdot,w)\) is affine for every \(w\).) Then the Isaacs condition
\eqref{eq:isaacs} holds.
\end{proposition}

\begin{proof}
For fixed \((x,p)\), the payoff \(c(x,u)+p\cdot F(x,u,\mu)\) is
continuous on the compact convex product \(U\times\Dcal\), quasiconvex in
\(u\), and affine, hence quasiconcave, in \(\mu\). Sion's theorem therefore
interchanges the infimum and supremum.
\end{proof}

\begin{corollary}\label{cor:value}
Let Assumption~\ref{ass:standing} hold and suppose the Isaacs condition
\eqref{eq:isaacs} holds (e.g., under the hypotheses of
Proposition~\ref{prop:isaacs}). Then
\begin{equation*}
  V^-(x)=V^+(x)=:V(x),\qquad x\in\R^d,
\end{equation*}
and \(V\in \BUC(\R^d)\) is the unique bounded uniformly continuous viscosity
solution of
\begin{equation}\label{eq:HJ-value}
  \rho V(x)-H(x,DV(x))=0.
\end{equation}
\end{corollary}

\begin{proof}
Both values solve \eqref{eq:HJ-value} by Theorem~\ref{thm:hji}; comparison
gives equality and uniqueness; see also Evans and Souganidis~\citep[Cor.~4.2(ii)]{EvansSouganidis}.
\end{proof}

\section{From Sampled DRO to the Continuous-Time Game}\label{sec:scheme}

We first discretize the averaged dynamics, then compare their limit with
that of the sampled recursion. For a step size \(\Delta\in(0,1]\), define
\begin{equation}\label{eq:TDelta}
  (T_\Delta v)(x)
  :=\inf_{u\in U}\sup_{\mu\in\Dcal}
  \bigl\{\Delta\,c(x,u)+e^{-\rho\Delta}\,v\bigl(x+\Delta F(x,u,\mu)\bigr)\bigr\},
  \qquad v\in \BUC(\R^d).
\end{equation}
This operator evaluates the value at the averaged next state; the sampled
recursion of Yang~\citep{YangDRO}, considered below, averages it over random
next states.

\begin{proposition}\label{prop:scheme-fixed-point}
Under Assumption~\ref{ass:standing}, for every \(\Delta\in(0,1]\):
\begin{enumerate}[label=(\roman*)]
  \item \(T_\Delta\) maps \(\BUC(\R^d)\) into itself, is monotone
  (\(v\le w\Rightarrow T_\Delta v\le T_\Delta w\)), and is a contraction with
  factor \(e^{-\rho\Delta}\) in the supremum norm. It therefore has a unique
  fixed point \(v_\Delta\in \BUC(\R^d)\), and
  \begin{equation}\label{eq:vDelta-bound}
    0\le v_\Delta\le
    \frac{\Delta M_c}{1-e^{-\rho\Delta}}\le \frac{e^{\rho}M_c}{\rho}.
  \end{equation}
  \item The family \(\{v_\Delta\}_{\Delta\in(0,1]}\) is equi-uniformly
  continuous:
  \begin{equation}\label{eq:vDelta-modulus}
    |v_\Delta(x)-v_\Delta(y)|\le e^{\rho}\,\omega_J(|x-y|),
    \qquad x,y\in\R^d,
  \end{equation}
  with \(\omega_J\) as in \eqref{eq:omega-J}.
\end{enumerate}
\end{proposition}

\begin{proof}
(i) Monotonicity and the contraction property are immediate from
\eqref{eq:TDelta}; if \(0\le v\le K\) then \(0\le T_\Delta v\le\Delta
M_c+e^{-\rho\Delta}K\), so the fixed point satisfies
\(v_\Delta\le\Delta M_c/(1-e^{-\rho\Delta})\), and
\(\Delta/(1-e^{-\rho\Delta})\le e^{\rho\Delta}/\rho\le e^{\rho}/\rho\)
by \(1-e^{-t}\ge te^{-t}\). That \(T_\Delta\) preserves \(\BUC\) follows from
\begin{equation}\label{eq:modulus-recursion}
  \omega_{T_\Delta v}(s)\le
  \Delta\,\omega_c(s)+e^{-\rho\Delta}\,\omega_v\bigl((1+\Delta L_F)s\bigr),
\end{equation}
where \(\omega_g\) denotes a modulus of \(g\); indeed
\(|x+\Delta F(x,u,\mu)-y-\Delta F(y,u,\mu)|\le(1+\Delta L_F)|x-y|\) by (A2).
Banach's fixed point theorem in the closed subspace \(\BUC(\R^d)\) of
\(C_b(\R^d)\) gives \(v_\Delta\).

(ii) Applying \eqref{eq:modulus-recursion} to the fixed point \(n\) times,
\[
  \omega_{v_\Delta}(s)
  \le \Delta\sum_{k=0}^{n-1}e^{-\rho k\Delta}
  \omega_c\bigl((1+\Delta L_F)^ks\bigr)
  +e^{-\rho n\Delta}\cdot 2\|v_\Delta\|_\infty.
\]
The last term tends to zero as \(n\to\infty\). On each interval
\([k\Delta,(k+1)\Delta]\), we have
\(e^{-\rho\tau}\ge e^{-\rho\Delta}e^{-\rho k\Delta}\) and
\(\omega_c(e^{L_F\tau}s)\ge\omega_c(e^{L_Fk\Delta}s)\).
Together with \((1+\Delta L_F)^k\le e^{L_Fk\Delta}\), these inequalities
bound the sum by an integral:
\[
  \omega_{v_\Delta}(s)
  \le e^{\rho\Delta}\int_0^\infty e^{-\rho\tau}
  \omega_c\bigl(e^{L_F\tau}s\bigr)\dd\tau
  =e^{\rho\Delta}\,\omega_J(s)\le e^{\rho}\,\omega_J(s).
\]
\end{proof}

\begin{theorem}\label{thm:scheme}
Under Assumption~\ref{ass:standing},
\begin{equation*}
  v_\Delta\longrightarrow V^+
  \qquad\text{locally uniformly on }\R^d\text{ as }\Delta\downarrow0.
\end{equation*}
If the Isaacs condition \eqref{eq:isaacs} holds, then \(v_\Delta\to V\), the
value of the game. The analogous scheme with the order of optimization
reversed, \(\sup_\mu\inf_u\), converges to \(V^-\).
\end{theorem}

\begin{proof}
\emph{Step 1 (consistency).} Let \(\varphi\in C^1(\R^d)\) and let
\(E\subset\R^d\) be compact. We claim
\begin{equation}\label{eq:consistency-scheme}
  \sup_{x\in E}
  \Bigl|\frac{\varphi(x)-T_\Delta\varphi(x)}{\Delta}
  -\bigl(\rho\varphi(x)-H^-(x,D\varphi(x))\bigr)\Bigr|
  \xrightarrow[\Delta\downarrow0]{}0.
\end{equation}
Choose \(R_E>0\) such that \(|x|\le R_E\) for every \(x\in E\).
Then \(|F(x,u,\mu)|\le C_F(1+R_E)\) on \(E\), so
\(\varphi(x+\Delta F)=\varphi(x)+\Delta D\varphi(x)\cdot F+\Delta\epsilon_\Delta\)
with \(\epsilon_\Delta\to0\) uniformly on \(E\times U\times\Dcal\) (uniform
continuity of \(D\varphi\) on a neighborhood of \(E\)), and
\(e^{-\rho\Delta}=1-\rho\Delta+O(\Delta^2)\). Substituting into
\eqref{eq:TDelta} and using
\(|\inf\sup A-\inf\sup B|\le\|A-B\|_\infty\),
\[
  T_\Delta\varphi(x)
  =\varphi(x)-\rho\Delta\varphi(x)
  +\Delta\inf_{u}\sup_{\mu}\{c(x,u)+D\varphi(x)\cdot F(x,u,\mu)\}
  +\Delta\,o_\Delta(1)
\]
uniformly on \(E\), which is \eqref{eq:consistency-scheme}.

\emph{Step 2 (compactness).} By
Proposition~\ref{prop:scheme-fixed-point}, the family \(\{v_\Delta\}\) is
equibounded and equi-uniformly continuous. By Arzel\`a--Ascoli and a diagonal
argument, every sequence \(\Delta_n\downarrow0\) admits a subsequence (not
relabeled) with \(v_{\Delta_n}\to v\) locally uniformly, where \(v\) is bounded
with modulus \(e^{\rho}\omega_J\); in particular \(v\in \BUC(\R^d)\).

\emph{Step 3 (the limit solves the PDE).} We show \(v\) is a viscosity
subsolution of \(\rho v-H^-(x,Dv)=0\); the supersolution property is
symmetric. Let \(v-\varphi\) have a local maximum at \(x_0\); replacing
\(\varphi\) by \(\varphi+|x-x_0|^2\), we may assume the maximum is strict on
some ball \(\bar B(x_0,r_0)\) and, as usual, \(v(x_0)=\varphi(x_0)\). Choose
maximizers \(x_n\) of \(v_{\Delta_n}-\varphi\) over \(\bar B(x_0,r_0)\); by
local uniform convergence and strictness, \(x_n\to x_0\) and
\(k_n:=(v_{\Delta_n}-\varphi)(x_n)\to0\). For \(n\) large,
\(\Delta_nC_F(1+|x_0|+r_0)<r_0/2\) and \(|x_n-x_0|<r_0/2\), so
\(x_n+\Delta_nF(x_n,u,\mu)\in\bar B(x_0,r_0)\) for all \((u,\mu)\), and on
that ball \(v_{\Delta_n}\le\varphi+k_n\). Monotonicity of \(T_{\Delta_n}\) and
the fixed-point relation give
\[
  \varphi(x_n)+k_n=v_{\Delta_n}(x_n)=T_{\Delta_n}v_{\Delta_n}(x_n)
  \le T_{\Delta_n}\varphi(x_n)+e^{-\rho\Delta_n}k_n,
\]
hence
\[
  \frac{\varphi(x_n)-T_{\Delta_n}\varphi(x_n)}{\Delta_n}
  \le -k_n\,\frac{1-e^{-\rho\Delta_n}}{\Delta_n}\xrightarrow[n\to\infty]{}0,
\]
since \(k_n\to0\) and \((1-e^{-\rho\Delta})/\Delta\to\rho\). By
\eqref{eq:consistency-scheme} evaluated along \(x_n\to x_0\), together with the
continuity of \(H^-\) (Lemma~\ref{lem:H-cont}) and of \(D\varphi\), the
left-hand side converges to \(\rho\varphi(x_0)-H^-(x_0,D\varphi(x_0))\).
Therefore \(\rho\varphi(x_0)-H^-(x_0,D\varphi(x_0))\le0\).

\emph{Step 4 (identification).} By Steps 2--3, every subsequential limit \(v\)
is a \(\BUC\) viscosity solution of \(\rho v-H^-(x,Dv)=0\). By
Theorem~\ref{thm:hji}, so is \(V^+\), and by Theorem~\ref{thm:comparison} the
\(\BUC\) solution is unique; hence \(v=V^+\) for every subsequential limit, and
the full family converges locally uniformly. Under \eqref{eq:isaacs},
\(V^+=V\) by Corollary~\ref{cor:value}. The statement for the reversed scheme
is proved identically with \(H^+\) and \(V^-\).
\end{proof}

\subsection{The Sampled Recursion Has the Same Limit}

The sampled operator averages the value at the random next state:
\begin{equation}\label{eq:TDelta-sampled}
  (\widetilde T_\Delta v)(x):=\inf_{u\in U}\sup_{\mu\in\Dcal}
  \left\{\Delta c(x,u)+e^{-\rho\Delta}
  \int_W v(x+\Delta f(x,u,w))\,\mu(\dd w)\right\}.
\end{equation}
For unbounded disturbances, we truncate at a threshold for which the probability
of a discarded jump is \(o(\Delta)\) and every retained displacement tends
to zero. The gap \(q_f<r\) permits both estimates
\eqref{eq:tail-K}--\eqref{eq:displacement}.

\begin{theorem}[Convergence of the Sampled Recursion]\label{thm:scheme-sampled}
Let Assumption~\ref{ass:standing} hold, and let \(F\) be induced by a continuous
\(f\) satisfying \eqref{eq:f-lip}--\eqref{eq:f-growth} with
\(q_f,q_L\in[0,r)\). Set \(m=m_r(\Dcal)\) and
\(L_{\mathrm{samp}}=L_f(1+m^{q_L/r})\), with \(m^0=1\).
Choose a bounded, nondecreasing, concave modulus \(\bar\omega_c\) of the
state dependence of \(c\), with \(\bar\omega_c(0)=0\) and
\(\bar\omega_c(h)\to0\) as \(h\downarrow0\), and define
\begin{equation*}
  \widetilde\omega_J(h):=\int_0^\infty e^{-\rho t}
       \bar\omega_c(e^{L_{\mathrm{samp}}t}h)\,\dd t.
\end{equation*}
For every \(\Delta\in(0,1]\), \(\widetilde T_\Delta\) preserves
\(\BUC(\R^d)\), is monotone, and is a contraction with factor
\(e^{-\rho\Delta}\). Its unique fixed point \(\widetilde v_\Delta\)
satisfies \eqref{eq:vDelta-bound} and
\begin{equation}\label{eq:sampled-modulus}
  |\widetilde v_\Delta(x)-\widetilde v_\Delta(y)|
  \le e^\rho\widetilde\omega_J(|x-y|).
\end{equation}
Moreover,
\begin{equation*}
  \widetilde v_\Delta\longrightarrow V^+
  \quad\text{locally uniformly on }\R^d\quad(\Delta\downarrow0).
\end{equation*}
Reversing the order of optimization gives convergence to \(V^-\).
Under the Isaacs condition, both limits equal \(V\).
\end{theorem}

\begin{proof}
\emph{Fixed points and a common modulus.}
Every bounded uniformly continuous function admits a bounded concave
majorant of its modulus. For such a modulus \(\eta\) of \(v\), using the
same disturbance in two states and then Jensen's inequality gives
\[
 |\widetilde T_\Delta v(x)-\widetilde T_\Delta v(y)|
 \le\Delta\bar\omega_c(|x-y|)+e^{-\rho\Delta}
       \eta((1+\Delta L_{\mathrm{samp}})|x-y|).
\]
Thus the operator maps \(\BUC\) into itself. Monotonicity and contraction
follow directly from \eqref{eq:TDelta-sampled}, and the invariant interval
argument in Proposition~\ref{prop:scheme-fixed-point} gives
\eqref{eq:vDelta-bound}.
To obtain a modulus uniform in \(\Delta\), start value iteration at
\(v_0=0\), \(v_{n+1}=\widetilde T_\Delta v_n\), and set
\[
 \eta_n(h)=\Delta\sum_{k=0}^{n-1}e^{-\rho k\Delta}
       \bar\omega_c((1+\Delta L_{\mathrm{samp}})^kh).
\]
Each \(\eta_n\) is concave, and induction using the preceding inequality
shows \(|v_n(x)-v_n(y)|\le\eta_n(|x-y|)\).
Uniform convergence of the iteration permits \(n\to\infty\).
Since \((1+\Delta L_{\mathrm{samp}})^k\le e^{L_{\mathrm{samp}}k\Delta}\),
comparison with the integral on each interval
\([k\Delta,(k+1)\Delta]\) yields
\(\eta_\infty(h)\le e^{\rho\Delta}\widetilde\omega_J(h)\).
Dominated convergence gives \(\widetilde\omega_J(h)\to0\) as
\(h\downarrow0\), proving \eqref{eq:sampled-modulus}.

\emph{Truncated consistency.}
Choose \(a\in(1/r,1/q_f)\), with \(1/q_f=\infty\) when \(q_f=0\), and
write \(K_\Delta=\Delta^{-a}\). Markov's inequality and the growth bound give
\begin{align}
 \sup_{\mu\in\Dcal}\mu(|w|>K_\Delta)&\le m\Delta^{ar}=o(\Delta),
       \label{eq:tail-K}\\
 \sup_{|x|\le R,u\in U,|w|\le K_\Delta}\Delta|f(x,u,w)|
 &\le C_f(1+R)\Delta(1+\Delta^{-aq_f})
       =:\delta_\Delta(R)\longrightarrow0.\label{eq:displacement}
\end{align}
Let \(\varphi\in C^1(\R^d)\) and let \(E\subset\bar B_R\) be compact.
On the truncated set the first-order Taylor error, divided by \(\Delta\),
is bounded after integration by
the product of \(\omega_{D\varphi}(\delta_\Delta(R))\) and \(C_f(1+R)(1+m^{q_f/r})\).
The omitted part of the drift is bounded by
\(C_f(1+R)m\) times \(\Delta^{ar}+\Delta^{a(r-q_f)}\), which also vanishes. Hence
\begin{equation}\label{eq:consistency-sampled}
\begin{aligned}
 &\sup_{x\in E,u\in U,\mu\in\Dcal}\left|\frac1\Delta
 \int_{|w|\le K_\Delta}[\varphi(x+\Delta f(x,u,w))-\varphi(x)]\,\dd\mu\right.\\
 &\hspace{35mm}\left.{}-D\varphi(x)\cdot F(x,u,\mu)\right|\longrightarrow0.
\end{aligned}
\end{equation}

\emph{Identification of the limit.}
The common bound and modulus yield locally uniform subsequential limits
\(v\in\BUC\). Suppose \(v-\varphi\) has a strict local maximum at
\(x_0\) on a closed ball, normalized by \(v(x_0)=\varphi(x_0)\).
For a convergent subsequence choose maximizers \(x_n\) of
\(\widetilde v_{\Delta_n}-\varphi\) on that ball and set
\(k_n=\widetilde v_{\Delta_n}(x_n)-\varphi(x_n)\). Then
\(x_n\to x_0\) and \(k_n\to0\). By \eqref{eq:displacement}, all truncated
next states remain in the ball for large \(n\).
Put \(B=e^\rho M_c/\rho\), \(p_n=m\Delta_n^{ar}\), and
\(A_n=\{|w|\le K_{\Delta_n}\}\). Using \eqref{eq:tail-K}, the fixed-point equation,
\(\widetilde v_{\Delta_n}\le\varphi+k_n\) at these next states, and
\(\|\widetilde v_{\Delta_n}\|_\infty\le B\), we obtain
\[
\begin{aligned}
 \varphi(x_n)+k_n\le\inf_u\sup_\mu\bigg\{\Delta_nc(x_n,u)
 &+e^{-\rho\Delta_n}\bigg[\int_{A_n}
 \varphi(x_n+\Delta_nf(x_n,u,w))\,\dd\mu\\
 &\hspace{26mm}+k_n\bigg]\bigg\}+(B+|k_n|)p_n.
\end{aligned}
\]
The term \(|k_n|p_n\) accounts for replacing \(k_n\mu(A_n)\) by \(k_n\)
without a sign assumption. From \eqref{eq:consistency-sampled} and
\(|\varphi(x_n)|\mu(A_n^c)=o(\Delta_n)\), the integral alone equals
\(\varphi(x_n)+\Delta_nD\varphi(x_n)\cdot F(x_n,u,\mu)+o(\Delta_n)\)
uniformly in \((u,\mu)\); \(k_n\) is kept as a separate term.
It follows that
\[
 k_n\frac{1-e^{-\rho\Delta_n}}{\Delta_n}
 \le-\rho\varphi(x_n)+H^-(x_n,D\varphi(x_n))+o(1).
\]
Letting \(n\to\infty\) proves the subsolution inequality. At a local
minimum, the same estimates with the inequalities reversed and the tail
bounded from below give the supersolution inequality. Comparison identifies
every subsequential limit with \(V^+\), proving convergence of the full
family. The reversed operator is treated in the same way with \(H^+\).
\end{proof}

\begin{remark}[Sampled Expansion and Scaling]\label{rem:scheme-interpretation}
The constant \(L_{\mathrm{samp}}\) controls state expansion before averaging.
It can exceed the best Lipschitz constant \(L_F\) of the relaxed drift,
because averaging can cancel samplewise expansion. The two modulus estimates
\eqref{eq:vDelta-modulus} and \eqref{eq:sampled-modulus} therefore use
different constants.
For a bounded smooth test function with bounded first derivative, the same
truncation argument gives
\[
 \int\varphi(x+\Delta f)\,\dd\mu
 =\varphi(x)+\Delta D\varphi(x)\cdot F(x,u,\mu)+o(\Delta)
\]
locally uniformly in \(x\), uniformly in \((u,\mu)\).
An \(O(\Delta^2)\) remainder additionally follows if the second derivative
of \(\varphi\) is bounded and \(\int|f(x,u,w)|^2\,\dd\mu\) is uniformly
bounded on compact state sets; the theorem does not impose this extra
moment condition. Centered increments of size \(\sqrt\Delta\), with
appropriate covariance and tightness assumptions, can instead yield a
second-order limit \citep{FlemingSouganidis}. This is a different regime.
\end{remark}

\section{Kantorovich Duality in the Hamiltonian}\label{sec:duality}

The limiting Hamiltonian still requires optimization over disturbance
laws. When \(f\) is nonlinear in the disturbance, the mean alone does not
determine this optimization. We apply static Wasserstein strong duality
\citep[Thm.~1]{BlanchetMurthy}, \citep[Thm.~1]{GaoKleywegt} to the inner
supremum in \(H^-\) and give a finite interval for the dual multiplier.
For \(\nu\in\Pcal_r(W)\), \(\theta\ge0\),
\(\Dcal=\Dcal_\theta(\nu)\), and \((x,u,p)\), define
\begin{equation*}
  \Phi_{x,u,p}(w):=p\cdot f(x,u,w).
\end{equation*}
When \(F\) is induced by \(f\),
\begin{equation*}
  p\cdot F(x,u,\mu)=\int_W \Phi_{x,u,p}(w)\,\mu(\dd w).
\end{equation*}

\begin{theorem}\label{thm:KD}
Let \(W\subset\R^\ell\) be closed, \(\nu\in\Pcal_r(W)\), \(\theta>0\), and let
\(\Phi:W\to\R\) be upper semicontinuous with the two-sided growth bound
\begin{equation}\label{eq:Phi-growth}
  |\Phi(w)|\le C(1+|w|^r),\qquad w\in W,
\end{equation}
for some constant \(C\ge0\). Write \(m_r(\nu):=\int_W|w|^r\,\nu(\dd w)\)
and, for \(\lambda\ge0\),
\begin{equation}\label{eq:dual-objective}
  D(\lambda):=\lambda\theta^r+
  \int_W \Phi^\lambda(w)\,\nu(\dd w),
  \qquad
  \Phi^\lambda(w):=\sup_{w'\in W}\bigl[\Phi(w')-\lambda|w'-w|^r\bigr].
\end{equation}
Then strong duality holds:
\begin{equation}\label{eq:KD}
  \sup_{\mu\in\Pcal_r(W):\,W_r(\mu,\nu)\le\theta}\int_W\Phi(w')\,\mu(\dd w')
  =
  \inf_{\lambda\ge0}D(\lambda),
\end{equation}
and both sides are finite. Moreover, \(D\) is well defined with values in
\((-\infty,+\infty]\), finite for \(\lambda\ge\lambda_0:=2^{r-1}C\), and
coercive,
\begin{equation}\label{eq:coercive}
  D(\lambda)\ \ge\ \lambda\theta^r-C\bigl(1+m_r(\nu)\bigr)
  \xrightarrow[\lambda\to\infty]{}+\infty,
\end{equation}
so the infimum in \eqref{eq:KD} may be restricted to the compact interval
\([0,\Lambda]\) with
\(\Lambda:=\bigl(D(\lambda_0)+C(1+m_r(\nu))\bigr)/\theta^r\ \ge\ \lambda_0\).
\end{theorem}

\begin{proof}
Apply Blanchet and Murthy~\citep[Thm.~1]{BlanchetMurthy} or Gao and Kleywegt~\citep[Thm.~1]{GaoKleywegt}
with transport cost \((w,w')\mapsto|w-w'|^r\). The space \(W\) is Polish,
the cost is lower semicontinuous, and \(\Phi\) is upper semicontinuous.
The growth bound \eqref{eq:Phi-growth} gives integrability against \(\nu\)
and finite growth relative to the transport cost. This proves \eqref{eq:KD}.
We now bound the primal value and the dual multiplier.

\emph{Finiteness of the primal.} Since \(\nu\) itself is feasible, the
left-hand side of \eqref{eq:KD} is at least
\(\int\Phi\,\dd\nu\ge-C(1+m_r(\nu))\); by \eqref{eq:Phi-growth} and the
uniform moment bound \eqref{eq:uniform-moment}, it is at most
\(C\bigl(1+(\theta+m_r(\nu)^{1/r})^r\bigr)<\infty\).

\emph{The dual objective.} For each \(\lambda\ge0\), taking \(w'=w\) in the
supremum gives \(\Phi^\lambda\ge\Phi\ge-C(1+|\cdot|^r)\in L^1(\nu)\), so
the integral in \eqref{eq:dual-objective} is well defined with values in
\((-\infty,+\infty]\), and \eqref{eq:coercive} follows; coercivity uses
\(\theta>0\). For measurability, \(\Phi^\lambda\) is lower semicontinuous
as a supremum of functions continuous in \(w\), hence Borel. By
\(|w'|^r\le2^{r-1}(|w|^r+|w'-w|^r)\),
\[
  \Phi(w')-\lambda|w'-w|^r
  \le C+C2^{r-1}|w|^r+(C2^{r-1}-\lambda)|w'-w|^r,
\]
so for \(\lambda\ge\lambda_0=2^{r-1}C\) one has
\(\Phi^\lambda(w)\le C+C2^{r-1}|w|^r\in L^1(\nu)\), whence
\(D(\lambda)<\infty\).

\emph{Compact multiplier interval.} If \(\lambda>\Lambda\), then by
\eqref{eq:coercive},
\(D(\lambda)>\Lambda\theta^r-C(1+m_r(\nu))=D(\lambda_0)\), so no such
\(\lambda\) improves on \(\lambda_0\in[0,\Lambda]\), and the infimum over
\([0,\infty)\) coincides with the infimum over \([0,\Lambda]\). Finally,
\(\Lambda\ge\lambda_0\) follows from
\(D(\lambda_0)\ge\lambda_0\theta^r-C(1+m_r(\nu))\).
\end{proof}

Consequently, whenever \(\Dcal=\Dcal_\theta(\nu)\) and \(F\) is induced by
\(f\),
\begin{equation*}
H^+(x,p)
=\sup_{\mu\in\Dcal}\inf_{u\in U}
\left\{c(x,u)+\int_Wp\cdot f(x,u,w)\,\mu(\dd w)\right\},
\end{equation*}
and, for \(\theta>0\), the inner measure supremum in \(H^-\) can be rewritten by \eqref{eq:KD}:
\begin{equation}\label{eq:Hminus-dual}
\begin{aligned}
H^-(x,p)=\inf_{u\in U,\lambda\ge0}\bigg\{c(x,u)&+\lambda\theta^r\\
&+\int_W\sup_{w'\in W}
\{p\cdot f(x,u,w')-\lambda|w'-w|^r\}\,\nu(\dd w)\bigg\},
\end{aligned}
\end{equation}
provided \(w\mapsto p\cdot f(x,u,w)\) satisfies the two-sided growth bound
\eqref{eq:Phi-growth}; this holds under \eqref{eq:f-growth} with
\(q_f\le r\), and the map is in fact continuous.
For \(\theta=0\), the ball is \(\{\nu\}\) and the measure supremum
is simply \(\int p\cdot f(x,u,w)\,\dd\nu\); the compact multiplier bound
in Theorem~\ref{thm:KD} is not used at zero radius.
Under the Isaacs condition (Proposition~\ref{prop:isaacs}), the same
representation applies to \(H=H^+=H^-\). Formula \eqref{eq:Hminus-dual} is the
continuous-time analogue of the semi-infinite Bellman reformulation in
discrete-time Wasserstein DRO \citep{EsfahaniKuhn,YangDRO}.

\subsection{Reduction for Affine Disturbance Dependence}\label{subsec:affine}

For affine disturbance dependence, the distributional optimization reduces
to a problem over means. Assume \(W=\R^\ell\) and
\begin{equation}\label{eq:affine-f}
  f(x,u,w)=f_0(x,u)+G(x,u)w,
\end{equation}
where \(G(x,u)\in\R^{d\times\ell}\). Let
\begin{equation*}
  \bar w_\nu:=\int_{\R^\ell}w\,\nu(\dd w).
\end{equation*}
Then
\begin{equation*}
  F(x,u,\mu)=f_0(x,u)+G(x,u)\bar w_\mu,
  \qquad
  \bar w_\mu:=\int_{\R^\ell}w\,\mu(\dd w).
\end{equation*}

\begin{proposition}\label{prop:affine-penalty}
Let \(\Dcal=\Dcal_\theta(\nu)\) on \(\R^\ell\), with \(r\ge1\). For
any \(q\in\R^\ell\),
\begin{equation}\label{eq:mean-support}
  \sup_{\mu\in\Dcal}\int q\cdot w\,\mu(\dd w)
  =q\cdot\bar w_\nu+\theta |q|.
\end{equation}
Therefore, for \(q=G(x,u)^{\mathsf T}p\),
\begin{equation}\label{eq:affine-H-inner}
  \sup_{\mu\in\Dcal}p\cdot F(x,u,\mu)
  =p\cdot f_0(x,u)+p\cdot G(x,u)\bar w_\nu
  +\theta |G(x,u)^{\mathsf T}p|.
\end{equation}
\end{proposition}

\begin{proof}
Let \(\kappa\in\Pi(\mu,\nu)\), with first coordinate \(w\sim\mu\) and second
coordinate \(w'\sim\nu\) as in \eqref{eq:wr}. Then, by H\"older's inequality
(\(\kappa\) is a probability measure and \(r\ge1\)),
\[
\begin{aligned}
  q\cdot(\bar w_\mu-\bar w_\nu)
  &=\int q\cdot(w-w')\,\kappa(\dd w,\dd w')\\
  &\le |q|\left(\int |w-w'|^r\,\kappa(\dd w,\dd w')\right)^{1/r}.
\end{aligned}
\]
Taking the infimum over transport plans gives
\(q\cdot(\bar w_\mu-\bar w_\nu)\le |q|W_r(\mu,\nu)\le\theta|q|\). Hence the
left-hand side of \eqref{eq:mean-support} is at most
\(q\cdot\bar w_\nu+\theta|q|\). If \(q\ne0\), let
\(e=q/|q|\) and define \(T(w)=w+\theta e\). The law
\(\mu=T_\#\nu\) satisfies \(W_r(\mu,\nu)\le\theta\) and
\(\int q\cdot w\,\mu(\dd w)=q\cdot\bar w_\nu+\theta|q|\). If \(q=0\), the
identity is trivial. This proves \eqref{eq:mean-support}. The formula
\eqref{eq:affine-H-inner} follows by taking \(q=G(x,u)^{\mathsf T}p\).
\end{proof}

\begin{corollary}
\label{cor:mean-set}
Let \(\Dcal=\Dcal_\theta(\nu)\) on \(W=\R^\ell\) with \(r\ge1\). Then
\begin{equation}\label{eq:mean-set}
  \{\bar w_\mu:\mu\in\Dcal_\theta(\nu)\}
  =\bar B(\bar w_\nu,\theta)
  :=\{z\in\R^\ell:|z-\bar w_\nu|\le\theta\}.
\end{equation}
Consequently, in the affine model \eqref{eq:affine-f}, both Hamiltonians
coincide with those of the finite-dimensional deterministic differential game
in which the adversary chooses a vector \(z\in\bar B(\bar w_\nu,\theta)\) and
the dynamics are \(\dot X_s=f_0(X_s,u_s)+G(X_s,u_s)z_s\):
\begin{align}
  H^+(x,p)&=\sup_{|e|\le\theta}\inf_{u\in U}
  \{c(x,u)+p\cdot f_0(x,u)+p\cdot G(x,u)(\bar w_\nu+e)\},
  \label{eq:Hplus-affine-fd}\\
  H^-(x,p)&=\inf_{u\in U}\sup_{|e|\le\theta}
  \{c(x,u)+p\cdot f_0(x,u)+p\cdot G(x,u)(\bar w_\nu+e)\}.
  \label{eq:Hminus-affine-fd}
\end{align}
\end{corollary}

\begin{proof}
The inclusion \(\supseteq\) in \eqref{eq:mean-set} follows from the shift
construction in the proof of Proposition~\ref{prop:affine-penalty}: for
\(|e'|\le\theta\), the law \((\,\cdot+e')_\#\nu\) lies in \(\Dcal_\theta(\nu)\)
and has mean \(\bar w_\nu+e'\). The inclusion \(\subseteq\) follows from
\eqref{eq:mean-support}: \(q\cdot(\bar w_\mu-\bar w_\nu)\le\theta|q|\) for all
\(q\) forces \(|\bar w_\mu-\bar w_\nu|\le\theta\). Since \(F\) and \(c\) depend
on \(\mu\) only through \(\bar w_\mu\) in the affine model, the identities
\eqref{eq:Hplus-affine-fd}--\eqref{eq:Hminus-affine-fd} follow.
\end{proof}

For \(r=1\), the mean map is Borel but may not be weakly continuous.
Projection by the mean and lifting by translations nevertheless preserve
measurable paths, nonanticipativity, and payoffs. The lower and upper values
therefore coincide with those of the finite-dimensional games, subject to
the latter's state and cost regularity assumptions. The details are in
Remark~\ref{rem:r-one}.

Thus, in the affine model the Hamiltonian associated with the upper value
becomes
\begin{equation}\label{eq:Hminus-affine}
  H^-(x,p)=\inf_{u\in U}
  \left\{
    c(x,u)+p\cdot f_0(x,u)+p\cdot G(x,u)\bar w_\nu
    +\theta |G(x,u)^{\mathsf T}p|
  \right\}.
\end{equation}
The last term is the support function of \(\bar B(0,\theta)\) at \(G(x,u)^{\mathsf T}p\). If, in addition, \(U\) is convex, \(c(x,\cdot)\) is convex,
\(f_0(x,\cdot)\) is affine, and \(u\mapsto G(x,u)\) is affine (in particular,
if \(G=G(x)\) does not depend on \(u\)), then \(f(x,\cdot,w)\) is affine for
every \(w\), Proposition~\ref{prop:isaacs} applies, and \(H^+=H^-\); the game
has a value by Corollary~\ref{cor:value}.

\section{Sensitivity, Certificates, and Statistical Consistency}
\label{sec:certificate}

We next ask how estimating the disturbance law affects the value.
Containment gives a cost certificate. Quantitative approximation requires
control of state sensitivity and movement of the ambiguity sets.
Write \(V^\pm_{\theta,\nu}\) and \(H^\pm_{\theta,\nu}\) for the
values and Hamiltonians associated with \(\Dcal_\theta(\nu)\).
Each ambiguity set must satisfy Assumption~\ref{ass:standing}, or admit the
affine reduction of Remark~\ref{rem:r-one} with its stated regularity.
Comparisons require common state and cost Lipschitz constants \(L_F,L_c\);
finite drift-growth constants \(C_F\) may differ.

\subsection{A Cost Bound under the True Disturbance Law}

Let \(\mu_*\in\Pcal_r(W)\) be the true law and \(\nu_N\) its empirical
estimate from independent samples. Whenever \(\mu_*\in\Dcal_\theta(\nu_N)\),
the constant path \(b(s)\equiv\mu_*\) is available to the adversary.
The following certificate is a direct consequence of this inclusion.

\begin{proposition}\label{prop:certificate}
Let \(\Dcal=\Dcal_\theta(\nu_N)\), and assume \(\mu_*\in\Dcal\). For every
controller strategy \(\alpha\in\Gammac\),
\begin{equation}\label{eq:certificate-alpha}
  \sup_{b\in\Bcal}J(x;\alpha[b],b)
  \ge
  J(x;\alpha[b^*],b^*),
  \qquad b^*(s)\equiv\mu_*.
\end{equation}
In particular, if \(\alpha^\varepsilon\) is \(\varepsilon\)-optimal for
\(V^-_{\theta,\nu_N}\), then
\begin{equation}\label{eq:certificate-value}
  J(x;\alpha^\varepsilon[b^*],b^*)
  \le V^-_{\theta,\nu_N}(x)+\varepsilon.
\end{equation}
Here \(\alpha^\varepsilon[b^*]\) is the control path generated by the
strategy \(\alpha^\varepsilon\) against the constant adversarial path
\(b^*\); it may depend on \(\mu_*\) (see
Remark~\ref{rem:certificate-info}).
\end{proposition}

\begin{proof}
Since \(b^*\in\Bcal\), \eqref{eq:certificate-alpha} follows directly from the
supremum over \(b\). If \(\alpha^\varepsilon\) is \(\varepsilon\)-optimal, then
\(\sup_bJ(x;\alpha^\varepsilon[b],b)\le V^-_{\theta,\nu_N}(x)+\varepsilon\),
which gives \eqref{eq:certificate-value}.
\end{proof}

\begin{remark}[Information and Implementable Controllers]
\label{rem:certificate-info}
The bound evaluates a strategy under the drift averaged with respect to the true law.
Because an Elliott--Kalton strategy observes the current adversarial law,
\(\alpha^\varepsilon[b^*]\) may require knowledge of \(\mu_*\). A state
feedback inherits the bound only if its closed loop is well posed for every
admissible \(b\) and its induced strategy has worst-case cost at most
\(V^-+\varepsilon\). Affinity in the disturbance is insufficient to ensure
these properties or a bang-bang minimizer. For instance, \(f=u+w\) with running cost
\(u^2\) is optimized by the interior action \(u=0\). Likewise,
\(\dot x=-\operatorname{sign}(x)+m\), \(x(0)=0\), with
\(0<m<1\) and \(\operatorname{sign}(0)=0\), has no Carath\'eodory solution.
The scalar experiments use a sample-and-hold policy, which is well posed
on each holding interval. Its costs are compared with the game values
without assuming \(\varepsilon\)-optimality.
\end{remark}

\subsection{Lipschitz Values}

\begin{assumption}[Lipschitz cost and discount condition]\label{ass:lip}
In addition to Assumption~\ref{ass:standing},
\begin{enumerate}[label=(A\arabic*),start=5]
  \item there exists \(L_c>0\) with \(|c(x,u)-c(y,u)|\le L_c|x-y|\) for all
  \(x,y\in\R^d\), \(u\in U\), and the discount rate satisfies \(\rho>L_F\).
\end{enumerate}
\end{assumption}

\begin{proposition}\label{prop:V-lip}
Under Assumptions~\ref{ass:standing} and \ref{ass:lip}, the values \(V^\pm\)
are Lipschitz continuous with constant
\begin{equation}\label{eq:LV}
  L_V:=\frac{L_c}{\rho-L_F}.
\end{equation}
\end{proposition}

\begin{proof}
By \eqref{eq:state-lip-initial} and (A5),
\(|J(x;a,b)-J(y;a,b)|\le L_c|x-y|\int_0^\infty
 e^{-(\rho-L_F)s}\,\dd s=L_V|x-y|\).
Taking the corresponding suprema and infima gives the claim.
\end{proof}

\subsection{Stability of Wasserstein Balls and of the Values}

\begin{lemma}\label{lem:hausdorff}
Let \(W=\R^\ell\), \(r\ge1\), \(\nu,\nu'\in\Pcal_r(\R^\ell)\), and
\(\theta,\theta'\ge0\). Then, in the Hausdorff distance induced by \(W_r\),
\begin{equation*}
  \dH\bigl(\Dcal_\theta(\nu),\Dcal_{\theta'}(\nu')\bigr)
  \le |\theta-\theta'|+W_r(\nu,\nu').
\end{equation*}
\end{lemma}

\begin{proof}
By symmetry it suffices to show that for every \(\mu\in\Dcal_\theta(\nu)\)
there is \(\mu'\in\Dcal_{\theta'}(\nu')\) with
\(W_r(\mu,\mu')\le(\theta-\theta')_++W_r(\nu,\nu')\).

\emph{Step 1 (recentering).} Let \(\kappa_1\in\Pi(\mu,\nu)\) and
\(\kappa_2\in\Pi(\nu,\nu')\) be optimal plans, and let
\((w_\mu,w_\nu,w_{\nu'})\) be a gluing of \(\kappa_1\) and \(\kappa_2\) on a
common probability space. Define \(\mu_1\) as the law of
\(w_\mu+(w_{\nu'}-w_\nu)\). Then
\(W_r(\mu_1,\nu')\le\|(w_\mu+w_{\nu'}-w_\nu)-w_{\nu'}\|_{L^r}
=\|w_\mu-w_\nu\|_{L^r}=W_r(\mu,\nu)\le\theta\), and
\(W_r(\mu,\mu_1)\le\|w_{\nu'}-w_\nu\|_{L^r}=W_r(\nu,\nu')\).

\emph{Step 2 (radial contraction).} If \(\theta\le\theta'\), take
\(\mu'=\mu_1\). Otherwise set \(t:=\theta'/\theta\in[0,1)\), let
\(\pi\in\Pi(\nu',\mu_1)\) be optimal with coordinates \((w_{\nu'},w_{\mu_1})\),
and define \(\mu'\) as the law of \((1-t)w_{\nu'}+tw_{\mu_1}\). Then
\(W_r(\mu',\nu')\le t\,\|w_{\mu_1}-w_{\nu'}\|_{L^r}
= t\,W_r(\mu_1,\nu')\le t\theta=\theta'\), and
\(W_r(\mu',\mu_1)\le(1-t)W_r(\mu_1,\nu')\le\theta-\theta'\). Combining the two
steps with the triangle inequality gives the claim.
\end{proof}

\begin{theorem}[Joint radius and center stability]
\label{thm:sensitivity}
Let \(W=\R^\ell\), let \(F\) be induced by \(f\), and suppose that
\(f(x,u,\cdot)\) is \(L_w\)-Lipschitz on \(\R^\ell\), uniformly in
\((x,u)\). Fix \(\theta,\theta'\ge0\) and
\(\nu,\nu'\in\Pcal_r(\R^\ell)\). Assume that
Assumptions~\ref{ass:standing} and \ref{ass:lip} hold for both ambiguity
sets, or, in the affine model \eqref{eq:affine-f}, that both reduced finite-dimensional
games satisfy the corresponding continuity, state regularity, bounded-cost,
and discounting hypotheses (Remark~\ref{rem:r-one}). In either case, require
\emph{common constants} \((L_F,L_c)\) and \(\rho>L_F\); the finite growth
constants \(C_F\) may differ. Then,
\begin{equation}\label{eq:H-sensitivity}
  \bigl|H^\pm_{\theta,\nu}(x,p)-H^\pm_{\theta',\nu'}(x,p)\bigr|
  \le L_w|p|\,\bigl(|\theta-\theta'|+W_r(\nu,\nu')\bigr),
  \qquad (x,p)\in\R^d\times\R^d,
\end{equation}
and
\begin{equation}\label{eq:V-sensitivity}
  \bigl\|V^\pm_{\theta,\nu}-V^\pm_{\theta',\nu'}\bigr\|_\infty
  \le \frac{L_wL_c}{\rho(\rho-L_F)}
  \bigl(|\theta-\theta'|+W_r(\nu,\nu')\bigr).
\end{equation}
\end{theorem}

\begin{proof}
\emph{Hamiltonian estimate.} For \(\mu,\mu'\in\Pcal_r\), Lipschitz continuity
of \(f\) in \(w\) and the Kantorovich--Rubinstein inequality give
\[
  |p\cdot F(x,u,\mu)-p\cdot F(x,u,\mu')|
  \le|p|\,L_w\,W_1(\mu,\mu')\le|p|\,L_w\,W_r(\mu,\mu').
\]
Hence \(\mu\mapsto L(u,\mu):=c(x,u)+p\cdot F(x,u,\mu)\) is
\(L_w|p|\)-Lipschitz for \(W_r\), uniformly in \(u\). For sets
\(A,A'\subset\Pcal_r\) and an equi-Lipschitz family,
\(|\sup_A L-\sup_{A'}L|\le L_w|p|\,\dH(A,A')\), and the same bound survives
inner infima or outer infima over \(u\). Estimate \eqref{eq:H-sensitivity} then
follows from Lemma~\ref{lem:hausdorff}.

\emph{Value estimate.} In the affine case, apply Proposition~\ref{prop:V-lip}
and the comparison argument below to the equivalent games of
Remark~\ref{rem:r-one}. Take \((V,H)=(V^-_{\theta,\nu},H^+_{\theta,\nu})\)
or \((V,H)=(V^+_{\theta,\nu},H^-_{\theta,\nu})\), with \((V',H')\)
corresponding to \((\theta',\nu')\). Set
\(\kappa:=L_wL_V(|\theta-\theta'|+W_r(\nu,\nu'))\), with the common
\(L_V\) from \eqref{eq:LV}. To show that \(V'+\kappa/\rho\) is a viscosity
supersolution of \(\rho v-H(x,Dv)=0\), let \(\varphi\in C^1\) touch \(V'+\kappa/\rho\) from
below at \(x_0\); then \(\psi:=\varphi-\kappa/\rho\) touches \(V'\) from below
at \(x_0\), so
\(\rho\psi(x_0)-H'(x_0,D\psi(x_0))\ge0\). Since \(V'\) is
\(L_V\)-Lipschitz (Proposition~\ref{prop:V-lip}), any such test function
satisfies \(|D\psi(x_0)|\le L_V\); hence, by \eqref{eq:H-sensitivity},
\(|H(x_0,D\varphi(x_0))-H'(x_0,D\psi(x_0))|\le\kappa\) (note
\(D\varphi=D\psi\)), and therefore
\[
\begin{aligned}
  \rho\varphi(x_0)-H(x_0,D\varphi(x_0))
  &=\rho\psi(x_0)+\kappa-H(x_0,D\psi(x_0))\\
  &\ge \rho\psi(x_0)-H'(x_0,D\psi(x_0))\ge0.
\end{aligned}
\]
By Theorem~\ref{thm:comparison} (comparison for \(H\)) applied to the
subsolution \(V\) and the supersolution \(V'+\kappa/\rho\),
\(V\le V'+\kappa/\rho\). Exchanging the roles of the two problems gives
\(\|V-V'\|_\infty\le\kappa/\rho\), which is \eqref{eq:V-sensitivity} after
inserting \eqref{eq:LV}.
\end{proof}

\begin{remark}
The values are nondecreasing in the ambiguity radius: since
\(\Dcal_\theta(\nu)\subset\Dcal_{\theta'}(\nu)\) for \(\theta\le\theta'\), the
Hamiltonians \(H^\pm_{\theta,\nu}\), and hence the values
\(V^\pm_{\theta,\nu}\), increase with \(\theta\).
Estimate \eqref{eq:V-sensitivity} bounds this increase linearly. The scalar
experiment illustrates this dependence; sharpness of the constant is not
established. The estimate is a dynamic counterpart of the static sensitivity
results of Bartl et al.~\citep{BartlEtAl}.
\end{remark}

\subsection{Statistical Consistency of the Data-Driven Value}

Let \(\mu_*\in\Pcal_r(\R^\ell)\) be the true law, let
\(\widehat w_1,\ldots,\widehat w_N\) be i.i.d. samples of \(\mu_*\), and let
\(\nu_N\) be the empirical measure \eqref{eq:empirical}. For example, assume
\(\int e^{\gamma|w|^a}\,\mu_*(\dd w)<\infty\) for some \(a>r\) and
\(\gamma>0\). The bounds of Fournier and Guillin~\citep[Thm.~2]{FournierGuillin}, given the
moment-bound constants, yield radii \(\theta_N(\beta)\downarrow0\) for
each \(\beta\in(0,1)\) such that
\begin{equation}\label{eq:concentration}
  \mathbb P\bigl\{W_r(\mu_*,\nu_N)\le\theta_N(\beta)\bigr\}\ge1-\beta.
\end{equation}
Write \(V^\pm_{0,\mu_*}\) for the value of the nominal relaxed problem, i.e.,
the (one-player) optimal control problem with drift \(F(x,u,\mu_*)\), obtained
by taking \(\Dcal=\{\mu_*\}\).

\begin{corollary}
\label{cor:statistical}
Suppose the hypotheses of Theorem~\ref{thm:sensitivity}, including its
affine-reduction alternative, hold with common constants \((L_F,L_c)\), \(\rho>L_F\), uniformly over the family
\(\{\Dcal_{\theta_N(\beta)}(\nu_N)\}_{N\in\N}\cup\{\{\mu_*\}\}\)
and that \eqref{eq:concentration} holds. Choose
\(\theta=\theta_N(\beta)\). For each fixed \(N\), with probability at least
\(1-\beta\):
\begin{enumerate}[label=(\roman*)]
  \item \(\mu_*\in\Dcal_{\theta_N}(\nu_N)\), so the certificate
  \eqref{eq:certificate-alpha}--\eqref{eq:certificate-value} holds; and
  \item the data-driven robust value satisfies the quantitative approximation bound:
  \begin{equation*}
    \bigl\|V^\pm_{\theta_N,\nu_N}-V^\pm_{0,\mu_*}\bigr\|_\infty
    \le \frac{2L_wL_c}{\rho(\rho-L_F)}\,\theta_N(\beta).
  \end{equation*}
\end{enumerate}
\end{corollary}

\begin{proof}
Part (i) is \eqref{eq:concentration} and
Proposition~\ref{prop:certificate}. For (ii), on the event
\(\{W_r(\mu_*,\nu_N)\le\theta_N\}\), apply \eqref{eq:V-sensitivity} with
\((\theta,\nu)=(\theta_N,\nu_N)\) and \((\theta',\nu')=(0,\mu_*)\):
\(|\theta_N-0|+W_r(\nu_N,\mu_*)\le2\theta_N\).
\end{proof}

\section{Numerical Illustrations}\label{sec:experiments}
Code and results are available on GitHub (Data Availability). Errors
are measured on the stated evaluation grids, and rates are empirical.
The convergence theorems concern continuous state and action spaces;
interpolation, finite action sets, and domain truncation introduce additional
errors, assessed here and in Appendix~\ref{app:experiments} through refinement checks.

\subsection{Scalar Reference Model}\label{subsec:scalar-reference}
\begin{example}[A scalar robust control problem]\label{ex:scalar}
Let \(d=\ell=1\), \(W=\R\), \(U=[-1,1]\), \(f(x,u,w)=u+w\), and
\(c(x,u)=g(x)\) with \(g\) bounded and Lipschitz. Let \(\nu\) have mean
\(\bar w\). Since \(W=\R\) is unbounded and \(f\) has linear growth in \(w\)
(\(q_f=1\)), Lemma~\ref{lem:F-regular} applies when \(r>1\)
(Remark~\ref{rem:strict}). For \(r=1\), we instead use the finite-dimensional
reformulation of Remark~\ref{rem:r-one}. The scalar relaxed-game
Hamiltonian and value formulas below remain valid under this reformulation;
it does not extend the sampled-recursion theorem to general \(W_1\)-balls.
Then \(f_0=u\), \(G\equiv1\), and by
\eqref{eq:Hminus-affine} and Proposition~\ref{prop:isaacs} (all hypotheses
hold: \(U\) convex, \(c\) independent of \(u\), \(f\) affine in \(u\), \(G\)
constant),
\[
  H(x,p)=g(x)+\bar wp+\theta|p|+\min_{|u|\le1}up
  =g(x)+\bar wp-(1-\theta)|p|.
\]
For \(\theta<1\), the value \(V\) is therefore the unique \(\BUC\) solution of
\[
  \rho V=g(x)+\bar wV'-(1-\theta)|V'|,
\]
which is the HJB equation of the \emph{one-player} control problem
\(\dot X=\bar w+\tilde u\), \(|\tilde u|\le1-\theta\): Wasserstein ambiguity of
radius \(\theta\) acts exactly as a shrinkage of the control authority from
\(1\) to \(1-\theta\). At \(\theta=1\), the \(|V'|\) term vanishes,
leaving \(\rho V=g+\bar wV'\). For \(\theta>1\), the coefficient of
\(|V'|\) is positive:
\(\rho V=g+\bar wV'+(\theta-1)|V'|\), the HJB equation of a maximizing
disturbance of magnitude \(\theta-1\). The Hamiltonian radius bound in Theorem~\ref{thm:sensitivity} is exact in
this example, since \(\partial_\theta H=|p|\). This does not imply that the
constant in the value-function bound is optimal.
\end{example}

In Example~\ref{ex:scalar}, take \(g(x)=x^2/(1+x^2)\), \(\rho=1\),
and \(\bar w=0\), so \(L_c=\max|g'|=9/(8\sqrt3)\approx0.650\).
The effective control authority vanishes at \(\theta=1\).
For \(\theta\in[0,1)\), writing \(\kappa:=1-\theta\), the value is
\begin{equation}\label{eq:exact-scalar}
  V_\theta(x)
  =\frac1\kappa\int_0^{|x|}e^{-\rho(|x|-y)/\kappa}\,g(y)\,\dd y,
\end{equation}
while \(V_1=g/\rho\) and, for \(\theta>1\),
\(V_\theta(x)=\int_0^\infty e^{-\rho s}g(|x|+(\theta-1)s)\,\dd s\). In each
case a direct computation shows that the formula defines a viscosity
solution of \eqref{eq:HJ-value}, so it is the value by
Theorem~\ref{thm:comparison}.
Numerical evaluations across the transition and implementation details are
given in Appendix~\ref{subsec:exp-scalar}.

\subsection{Sampling Effects and the Common Continuous-Time Limit}
\label{subsec:exp-scheme}

At a fixed step, the sampled recursion retains the effect of disturbance
spread. We first benchmark the relaxed fixed point \(v_\Delta\) of
\eqref{eq:TDelta} against \(V_{0.5}\) in \eqref{eq:exact-scalar}.
The maximum grid error on \([-2,2]\) decreases from
\(1.73\times10^{-1}\) to \(1.01\times10^{-2}\) as
\(\Delta\) decreases from 0.4 to 0.025, with an observed convergence rate of approximately 1.02.
Appendix~\ref{subsec:exp-scalar} and Table~\ref{tab:scheme} specify the grid and boundary values taken from the reference solution.
This observed rate complements the qualitative convergence result of
Theorem~\ref{thm:scheme}.

To isolate sampling effects in Theorem~\ref{thm:scheme-sampled}, let
\(\Dcal\) consist of the laws
\(\mu_{m,s}=\tfrac12\delta_{m-s}+\tfrac12\delta_{m+s}\),
\(m\in\{0,\pm\tfrac14,\pm\tfrac12\}\), and
\(s\in\{0,\tfrac12,1\}\). The fifteen laws have five means and variances
up to one. The relaxed recursion uses only the means, whose support
function agrees with that of \([-\tfrac12,\tfrac12]\); the sampled
recursion also depends on their spreads. Both use node-aligned control
grids and direct maximization over this finite family. Minimization over the continuous control set at steps 0.4, 0.2, and 0.1
found no improvement over the control-grid minimum at the evaluated states.
Figure~\ref{fig:bridge} shows the maximum gap between the computed values on
\([-2,2]\) decreasing from \(1.2\times10^{-1}\) at \(\Delta=0.4\) to
\(6.0\times10^{-3}\) at \(\Delta=0.025\), with an observed convergence rate of approximately 1.08. Both approach the same relaxed-game value \(V\)
at observed rates close to one, consistent with the higher-order local
sampling contribution in this bounded-support example.

\begin{figure}[t]
  \centering
  \includegraphics[width=.55\textwidth]{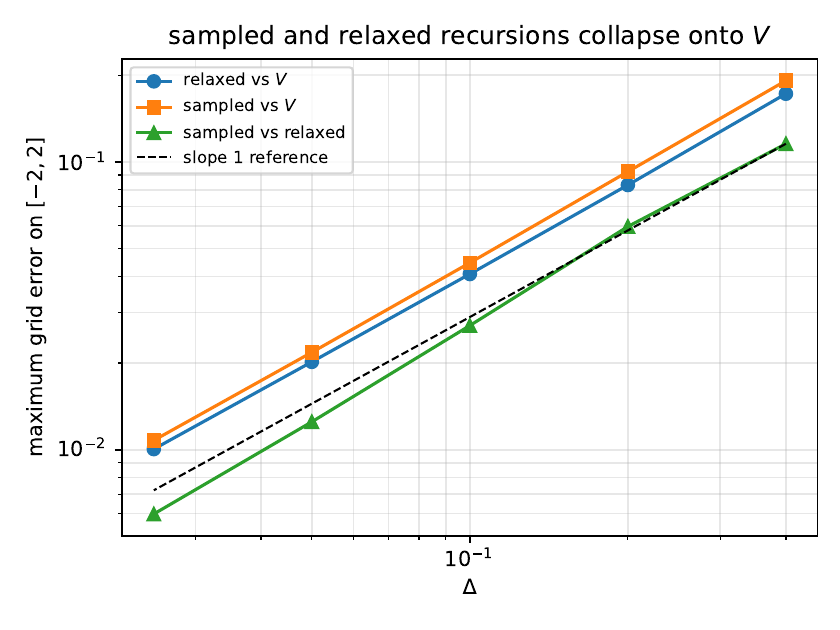}
  \caption{The discrete-to-continuous bridge
  (Theorem~\ref{thm:scheme-sampled}): maximum grid distances between approximations
  to the sampled fixed point \(\widetilde v_\Delta\), the relaxed fixed point
  \(v_\Delta\), and the value \(V\), for a finite ambiguity family with five
  distinct means and variances up to one. All three curves show approximately first-order decay.}
  \label{fig:bridge}
\end{figure}

\subsection{Data-Driven Value Approximation}\label{subsec:stat-summary}
At fixed continuous-time dynamics, we next vary the empirical law.
Conditional median value errors decay approximately as \(N^{-0.54}\)
over the tested sample sizes (Table~\ref{tab:consistency-main}).
For the scalar model, take the true law
\(\mu_*=\tfrac12N(-0.5,0.2^2)+\tfrac12N(0.7,0.3^2)\), with mean \(0.1\),
and use \(W_1\) ambiguity via the affine mean reduction
(Corollary~\ref{cor:mean-set} and Remark~\ref{rem:r-one}).
For each sample size, we generate 1000 empirical laws. The oracle radius
\(\theta_N\) is the empirical 90th percentile of the Wasserstein errors,
evaluated by analytical CDF integration over these replications. The bootstrap
radius \(\widehat\theta_N\) is the empirical 90th percentile
over 200 resamples of each empirical law. These radii are calibrated within the simulation. The finite-sample
guarantee of Corollary~\ref{cor:statistical} applies to radii satisfying
\eqref{eq:concentration}.

Table~\ref{tab:consistency-main} reports the median maximum value difference
on \([-6,6]\cap0.02\mathbb Z\), using replications satisfying
\(\max(\theta_N,\widehat\theta_N)+|\bar w_{\nu_N}|<0.97\).
This condition keeps every recovery speed above \(0.03\), retaining 993
replications at \(N=10\) and all 1000 otherwise. Bootstrap containment in
the final column uses all 1000 draws. The displayed bound is the theoretical
reference on the containment event in Corollary~\ref{cor:statistical}.
Appendix~\ref{subsec:exp-stat} gives the full protocol,
accuracy checks, and exploratory sample-and-hold cost comparisons.

\begin{table}[t]
\caption{Data-driven value approximation: conditional median grid error and
theoretical reference bound; bootstrap containment uses all replications}
\label{tab:consistency-main}
\centering\small
\begin{tabular}{rrrrrr}
\toprule
\(N\)&Retained&\(\theta_N\)&Median error&Bound&Containment\\
\midrule
10&993&0.360&0.159&0.467&81.5\%\\
30&1000&0.207&0.085&0.269&86.8\%\\
100&1000&0.111&0.042&0.144&89.9\%\\
300&1000&0.067&0.025&0.087&88.4\%\\
1000&1000&0.036&0.013&0.047&89.1\%\\
\bottomrule
\end{tabular}
\end{table}

\subsection{A Nonlinear Model: Limits of Uniform Shifts}\label{subsec:exp-nonaffine}
We test nonlinear transport with \(f(x,u,w)=u+\tanh(w)\), retaining
the scalar control set, cost, and discount; \(r=2\). The empirical law has
200 atoms: 140 drawn from \(N(1.5,0.2^2)\) and 60 from
\(N(0,0.15^2)\), using seed 21. Its nominal averaged disturbance is
approximately \(0.627\).

The drift satisfies Lemma~\ref{lem:F-regular} with \(q_f=q_L=0\).
The adversary acts through \(\int\tanh(w)\,\dd\mu\), whose attainable
interval \([m_-(\theta),m_+(\theta)]\) is computed by
Theorem~\ref{thm:KD}. For example,
\[
 m_+(\theta)=\inf_{\lambda\ge0}\left\{\lambda\theta^2+
 \frac1{200}\sum_i\sup_z[\tanh(z)-\lambda(z-\widehat w_i)^2]\right\}.
\]
\begin{figure}[!htbp]
\centering
\includegraphics[width=\textwidth]{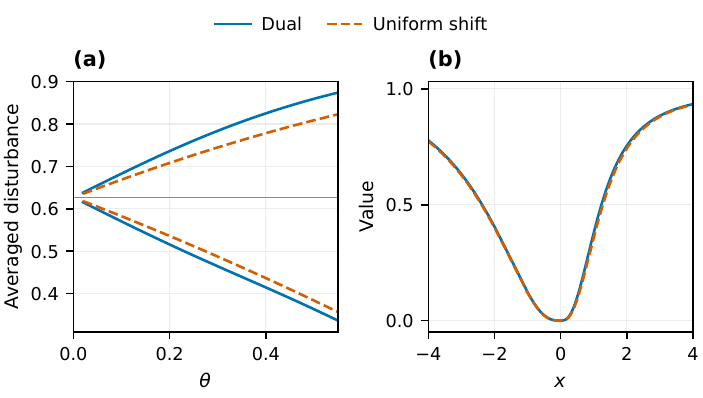}
\caption{Nonlinear model. (a) Uniform shifts give a narrower attainable
averaged-disturbance interval than distributional optimization; the gray
line is the nominal averaged disturbance. (b) At \(\theta=0.3\), uniform
shifts underestimate the robust value.}
\label{fig:nonaffine}
\end{figure}

At \(\theta=0.5\), distributional optimization gives a worst-case
right-side recovery speed of \(0.141\), compared with \(0.191\) for a
uniform shift. The latter overestimates the recovery speed by about 36\%
and understates the robust value. Table~\ref{tab:nonaffine} compares the
dual endpoints with the feasible uniform-shift interval
\([\frac1{200}\sum_i\tanh(\widehat w_i-\theta),
\frac1{200}\sum_i\tanh(\widehat w_i+\theta)]\).
We solve the continuous dual by stationary-point enumeration and scalar
minimization. The signed gaps against a feasible particle calculation are
below \(10^{-10}\) in absolute value.
Appendix~\ref{subsec:nonaffine-procedure} describes the solver and this
floating-point comparison.

\begin{table}[t]
\caption{Nonlinear model: dual and uniform-shift drift intervals, and
maximum value difference on \([-8,8]\cap0.0005\mathbb Z\)}
\label{tab:nonaffine}
\centering\small
\begin{tabular}{rccc}
\toprule
\(\theta\)&Dual interval&Uniform shift&Value difference\\
\midrule
0.1&\([0.570,0.683]\)&\([0.582,0.669]\)&\(6.8\times10^{-3}\)\\
0.2&\([0.516,0.736]\)&\([0.536,0.708]\)&\(1.4\times10^{-2}\)\\
0.3&\([0.464,0.784]\)&\([0.487,0.745]\)&\(2.1\times10^{-2}\)\\
0.5&\([0.363,0.859]\)&\([0.384,0.809]\)&\(2.9\times10^{-2}\)\\
\bottomrule
\end{tabular}
\end{table}

At \(\theta=0.5\), the 140 atoms from the higher-mean component receive approximately 13.0\% of
the squared transport budget for the upper endpoint and 50.5\% for the lower.

\subsection{Inventory Policies under Increasing Ambiguity}
\label{subsec:inventory-summary}
Finally, consider production decisions in the two-product model
\(\dot X=u-w\), with
\(u\in[0,1]^2\), \(w\in\R^2\), discount
\(\rho=1\), and cost
\(c(x,u)=\psi(x_1)+\psi(x_2)+0.25(u_1+u_2)\), where
\(\psi(t)=t_+/(1+t_+)+2t_-/(1+t_-)\).
A \(W_2\)-ball around the empirical demand law reduces to a mean ball
by Corollary~\ref{cor:mean-set}.

In Figure~\ref{fig:inventory}, at \(h=\Delta=0.0125\), increasing the radius
from 0 to 0.4 raises product 1 production at \((0,0.5)\) from 0.5087 to
0.7915, while its full-production band along \(x_2=0.5\) contracts from
\([-2.0375,-0.0125]\) to \([-1.4125,-0.0125]\).
These bands are extracted across the computational domain \([-5,6]^2\).
In deep backlogs, saturating costs reduce the marginal benefit of
replenishment, and discounting downweights later costs. Demand data and
space, time, action, and domain refinements are in
Appendix~\ref{subsec:exp-inventory}.

\begin{figure}[!htbp]
\centering
\includegraphics[width=\textwidth]{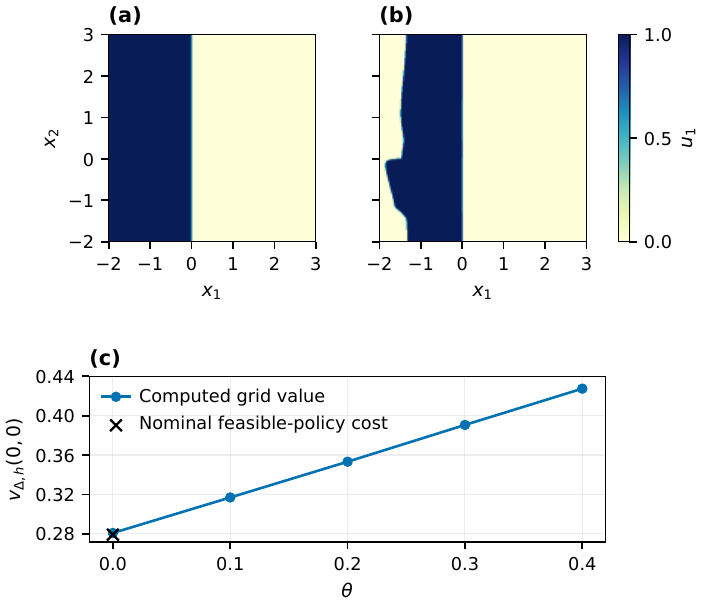}
\caption{Computed inventory policies and values at \(h=\Delta=0.0125\).
(a,b) Production maps for product 1 at radii 0 and 0.4, respectively, show
increased production near balance and a smaller full-production region.
(c) \(v_{\Delta,h}(0,0)\) versus radius, with the nominal
continuous-time feasible-policy cost marked separately.}
\label{fig:inventory}
\end{figure}

\FloatBarrier
\section{Modeling Comparisons}\label{sec:discussion}
\begin{remark}[The Relaxed-Drift Model]\label{rem:modeling}
For fixed action paths, the state equation is deterministic. The adversary
chooses a law in \(\Dcal\) at each time, as in classical relaxed control
\citep{Warga}. The pointwise constraint admits the concatenations used in
DPP; the following subsection considers a cumulative transport budget.
For fixed \(u\), a supremum of a linear expectation over all of
\(\Pcal(W)\) agrees with the supremum over pure disturbances, whenever the
expectations are defined. This identity does not extend through an inner
minimization: for \(U=W=\{-1,1\}\) and payoff \(uw\),
\(\sup_w\inf_u uw=-1\), whereas
\(\sup_{\mu\in\Pcal(W)}\inf_u u\int w\,\dd\mu=0\).
Thus relaxation can change the Hamiltonian and value of the lower game. The
information pattern determines how distributional optimization interacts
with the controller's choice.
\end{remark}
\subsection{Rectangular and Budget Constraints}
Constant paths taking values in \(\Dcal_\theta(\nu)\) satisfy the pointwise constraint.
For a fixed controller strategy, restricting the adversary to these paths
can only decrease the worst-case cost and preserves
Proposition~\ref{prop:certificate}. Comparing optimized values also requires
matching the controller's information.

A discounted transport budget gives the class
\begin{equation*}
 \Bcal_\Theta=\left\{b:[0,\infty)\to\Pcal_r(W)\text{ measurable}:
 \int_0^\infty e^{-\rho s}W_r^r(b(s),\nu)\,\dd s\le\Theta\right\}.
\end{equation*}
A budget game must restrict attention to paths with well-posed state
equations; the standing hypotheses ensure this only for the rectangular
class.
If \(\Theta\ge\theta^r/\rho\), then
\(\Bcal\subseteq\Bcal_\Theta\). For a controller defined on both classes,
its worst-case budget cost is consequently no smaller. On \(W=\R^\ell\),
when \(\theta>0\) and \(0<\Theta<\theta^r/\rho\), neither class contains
the other: a constant translation at distance \(\theta\) violates the
budget, while a translation at distance \(M>\theta\) during a sufficiently
short interval satisfies the budget but violates the pointwise constraint.
At \(\Theta=0\) only the nominal law is admissible a.e.; on restricted
supports the translations used above may be unavailable.

Noncontainment does not preclude an ordering of values in a particular
problem. For example, take no controller, \(f=w\), \(\nu=\delta_0\),
\(r=2\), \(\rho=1\), \(x_0=0\), and \(c(x)=\min(|x|,1)\).
With \(\theta=1\), the admissible path \(b(s)=\delta_1\) has cost
\(1-e^{-1}\). Every path satisfying the budget \(\Theta\) has cost at most
\(\int_0^\infty e^{-s}|\int w\,b(s)(\dd w)|\,\dd s\le\sqrt\Theta\),
by Fubini and Cauchy--Schwarz. At \(\Theta=0.01\), the rectangular value
therefore exceeds the budget value. Comparisons for other budgets depend
on the admissible-path inclusion and the controller's information.

\subsection{Velocity Sets and the Information Pattern}
The set \(K(x,u)=\{F(x,u,\mu):\mu\in\Dcal\}\) is compact and is convex
when \(F\) is affine in \(\mu\) and \(\Dcal\) is convex. For a fixed
measurable control, a relaxed trajectory satisfies
\(\dot X_s\in K(X_s,a(s))\) a.e. Conversely, an absolutely continuous
trajectory satisfying this inclusion admits a measurable law realizing its
velocity: apply measurable selection to the nonempty closed fibers of the
continuous action map, after choosing measurable representatives of
\(a\) and \(\dot X\). The equality of \(H^-\) with
\(\inf_u\{c(x,u)+\sup_{v\in K(x,u)}p\cdot v\}\) is immediate.

The sets \(K(x,u)\) alone do not preserve \(H^+\): they discard the way
one common law couples the velocities at different controls. For
\(U=W=\{-1,1\}\), \(\Dcal=\Pcal(W)\), \(f=uw\), \(c=0\), and
\(p=1\), the original \(H^+\) is zero. Both velocity sets are
\([-1,1]\), but an adversary choosing a control-independent velocity
\(v\in[-1,1]\) would give \(\sup_v\inf_u v=1\).
The mean-vector reduction in Corollary~\ref{cor:mean-set} retains one
common vector action across controls. This is the coupling lost when only
the separate sets \(K(x,u)\) are specified.

\subsection{Scope of the Numerical Conclusions}
The convergence theorems apply before spatial interpolation, domain
truncation, and action discretization. The refinements in
Appendix~\ref{subsec:exp-inventory} examine these additional errors for the
reported computations. A general rate for fully discrete schemes and
closed-loop guarantees for approximate feedback policies require further
analysis.

\subsection{Affine Reductions and Penalties}
\begin{remark}[The Case \(r=1\) and Equivalence of the Games]\label{rem:r-one}
The mean-set identity holds for every \(r\ge1\). Although the mean map
\(M(\mu)=\int w\,\dd\mu\) need not be weakly continuous for \(r=1\),
it is Borel: its coordinates are limits of integrals of bounded continuous
truncations. The map
\(S(z)=(w\mapsto w+z-\bar w_\nu)_\#\nu\) is weakly continuous,
takes \(\bar B(\bar w_\nu,\theta)\) into \(\Dcal_\theta(\nu)\), and
satisfies \(M(S(z))=z\). Applying these maps pointwise to paths preserves
measurability and a.e. equality. Composition therefore projects and lifts
nonanticipative strategies, with identical drifts and costs. More explicitly,
a controller strategy on laws can be restricted along \(S\), and one on
mean paths can be lifted along \(M\); an adversary strategy is projected by
\(M\) and lifted by \(S\). Taking infima and suprema in both directions
proves equality of the corresponding lower and upper values.
The finite-dimensional game therefore applies also at \(r=1\), provided
\((x,u,z)\mapsto f_0(x,u)+G(x,u)z\) and \(c\) satisfy the continuity,
state Lipschitz, growth, and bounded-cost hypotheses. The mean ball supplies
the compact action set. The sampled recursion still depends on the full
increment law and requires the moment conditions in
Theorem~\ref{thm:scheme-sampled}.
\end{remark}
\begin{remark}[Comparison with \(H^\infty\) and risk-sensitive penalties]
\label{rem:hinfty}
Quadratic energy penalties in \(H^\infty\) control \citep{BasarBernhard}
give a smooth correction to the Hamiltonian. For the normalization
\(\sup_e\{p\cdot G(x,u)e-\gamma^2|e|^2/2\}\), this correction is
\(\frac{1}{2\gamma^2}|G(x,u)^{\mathsf T}p|^2\); related
risk-sensitive formulations depend on the model and normalization.
A Wasserstein constraint on the disturbance law yields the positively
homogeneous, non-smooth penalty
\(\theta|G(x,u)^{\mathsf T}p|\) in \eqref{eq:Hminus-affine}. The radius \(\theta\) can be calibrated by measure concentration
(Section~\ref{sec:certificate}); \(\gamma\) specifies the energy penalty. As \(\theta\downarrow0\) the nominal
Hamiltonian is recovered locally uniformly in the state and gradient
variables. Uniformity over the whole state space and bounded gradient sets
holds under the additional, state-independent disturbance Lipschitz bound of
Theorem~\ref{thm:sensitivity}.
\end{remark}
\subsection{Disturbance-Dependent Costs}
\begin{remark}[Disturbance-Dependent Running Costs]\label{rem:cost-mu}
The dynamic programming, viscosity, and comparison arguments also apply to
\(C(x,u,\mu)=\int\tilde c(x,u,w)\,\dd\mu\) if this averaged cost is
bounded, jointly continuous, and uniformly continuous in the state uniformly
over both actions. Sufficient primitive assumptions are joint continuity,
\(|\tilde c(x,u,w)|\le C_0(1+|w|^{q_c})\) with \(q_c<r\), and
\(|\tilde c(x,u,w)-\tilde c(y,u,w)|\le
\omega(|x-y|)(1+|w|^{q_c})\), with \(\omega(h)\to0\).
The Isaacs condition must then be checked for \(C+p\cdot F\).
The dual integrand becomes \(\tilde c+p\cdot f\). A nonlinear disturbance
cost can retain dependence on the full law even when the drift is affine. In the Lipschitz setting of Theorem~\ref{thm:sensitivity},
if \(\tilde c\) is also \(L_{c,w}\)-Lipschitz in \(w\), the corresponding
bounds contain the extra cost term:
\( |H-H'|\le(L_{c,w}+L_w|p|)d\) and
\(\|V-V'\|_\infty\le(L_{c,w}+L_wL_V)d/\rho\), where
\(d=|\theta-\theta'|+W_r(\nu,\nu')\) and the averaged state-cost
Lipschitz constant used to define \(L_V\) is common to both problems.
We retain \(c=c(x,u)\) in all stated results and experiments.
\end{remark}

\section{Conclusions}
The sampled limit covers both orders of optimization and unbounded disturbances
under subcritical drift growth. Joint radius and center stability transfers
distributional estimation bounds to values. Experiments illustrate exact
affine mean reduction, nonlinear transport allocation effects, and contrasting
inventory responses to ambiguity near balance and in deep backlogs. Fully
discrete error rates and value guarantees for implementable feedback controls
remain open under the present assumptions.

\section*{Declarations}
\begingroup
\setlength{\parindent}{0pt}
\setlength{\parskip}{0.5\baselineskip}
\textbf{Funding.} Yeoneung Kim is supported by the National Research
Foundation of Korea (NRF) grants funded by the Korea government (MSIT)
(RS-2023-00219980 and RS-2026-25490291), and by Yonsei University.

\textbf{Competing Interests.} The authors have no relevant financial or
non-financial interests to disclose.

\textbf{Additional Experiments.}
Appendix~\ref{app:experiments} contains additional numerical experiments,
implementation details, and refinement checks.

\textbf{Data Availability.}
The code, synthetic data, and numerical results are publicly available at
\url{https://github.com/yeoneung/dro}. The data can be regenerated using
the supplied code and random seeds.

\textbf{Author Contributions.} Minjung Gim and Hee Jun Yang contributed
equally to this work.

\textbf{AI Assistance.} The authors used OpenAI Codex to assist with code verification, mathematical typesetting, and manuscript editing. The authors have verified all mathematical arguments and take full responsibility for the final manuscript.
\par
\endgroup

\clearpage
\appendix
\numberwithin{equation}{section}
\numberwithin{figure}{section}
\numberwithin{table}{section}
\input{appendix_experiments}
\clearpage
\bibliographystyle{spmpsci}
\begingroup
\normalfont
\let\emph\textup
\let\textit\textup
\urlstyle{same}
\bibliography{ref}
\endgroup

\end{document}

%% file: appendix_experiments.tex
\section{Additional Numerical Experiments}\label{app:experiments}
\suppressfloats[t]
The following sections give reference values, simulation protocols, and
refinement checks for Section~\ref{sec:experiments}.
Reported numerical errors are measured on the specified grids. All
computations use floating-point arithmetic; quadrature refinements and
primal--dual comparisons are numerical checks rather than interval bounds.
The analytic quadrature bound below is stated separately in exact arithmetic.

\subsection{Radius Dependence and Loss of Control Authority}
\label{subsec:exp-scalar}

We use the scalar reference model and value formulas in
Section~\ref{subsec:scalar-reference}. Figure~\ref{fig:scalar}(a) displays
\(V_\theta\) across the transition: for \(\theta<1\) the controller retains
authority \(1-\theta\) and \(V_\theta\) vanishes at the origin; at
\(\theta=1\) the gradient term vanishes and \(V_1=g/\rho\); for
\(\theta>1\) the coefficient of \(|V'|\) becomes positive. Figure~\ref{fig:scalar}(b) reports the maximum difference on 12000
equally spaced points in \([0,25]\). A linear regression over
\(\theta=0.05,0.10,\ldots,0.40\) has slope approximately \(0.31\).
The ratios of maximum difference to radius vary with the radius:
approximately \(0.2665\), \(0.3026\), and \(0.3560\) at
\(0.05\), \(0.4\), and \(0.9\), respectively. Thus linear dependence is
an approximation over the tested interval. The theoretical bound uses
\(L_w=1\), \(L_c=9/(8\sqrt3)\), and the limit \(L_F\downarrow0\);
the experiment does not establish a sharp value-bound constant.

\begin{figure}[t]
  \centering
  \includegraphics[width=\textwidth]{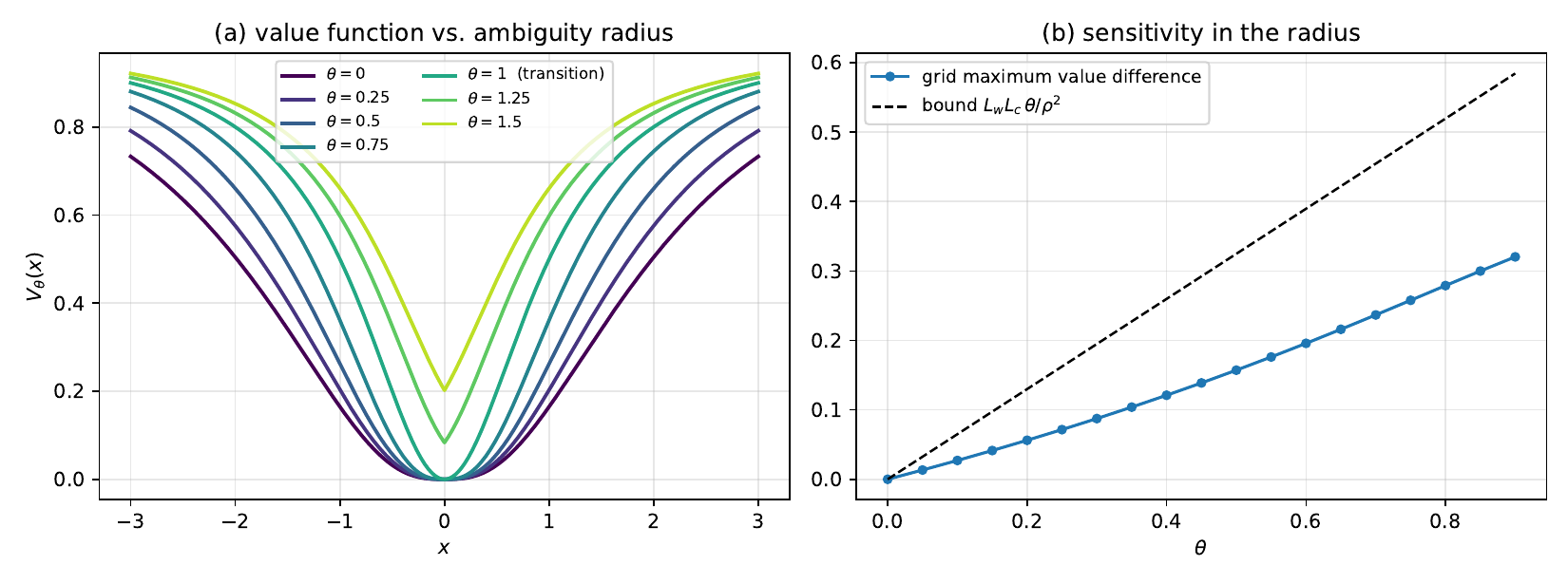}
  \caption{The scalar example of Section~\ref{subsec:exp-scalar}.
  (a)~Numerical evaluations of the value formulas \(V_\theta\) across the transition at
  \(\theta=1\). (b)~Maximum value difference on the evaluation grid versus radius;
  the regression slope over \([0.05,0.40]\) is \(0.31\), below the bound of
  Theorem~\ref{thm:sensitivity} (dashed).}
  \label{fig:scalar}
\end{figure}

\emph{Relaxed-scheme benchmark.}
For the test in Section~\ref{subsec:exp-scheme}, we
compare the fixed point of \eqref{eq:TDelta} at \(\theta=0.5\) with
\eqref{eq:exact-scalar}. Piecewise-linear grid iterates use sliding window extrema over node-aligned
actions, with \(\Delta/h\) and \(\Delta\theta/h\) integers and spatial
spacing \(h\) decreasing quadratically with \(\Delta\).
Outside \([-4,4]\), the reference formula supplies boundary values.
This benchmark measures discretization error with known boundary data;
domain truncation is assessed separately in the inventory experiment.
Table~\ref{tab:scheme} gives the maximum grid errors on \([-2,2]\).
The fitted convergence rate is approximately 1.02, complementing the qualitative
convergence in Theorem~\ref{thm:scheme}.

\begin{table}[t]
  \centering\small
  \caption{Convergence of the implemented relaxed scheme (\(\theta=0.5\)):
  maximum grid error on \([-2,2]\) against the value formula
  \eqref{eq:exact-scalar}. Least-squares convergence rate: 1.02.}
  \label{tab:scheme}
  \begin{tabular}{lccccc}
    \toprule
    \(\Delta\) & 0.4 & 0.2 & 0.1 & 0.05 & 0.025\\
    \midrule
    Grid error
    & \(1.73\times10^{-1}\) & \(8.31\times10^{-2}\) & \(4.08\times10^{-2}\)
    & \(2.02\times10^{-2}\) & \(1.01\times10^{-2}\)\\
    \bottomrule
  \end{tabular}
\end{table}

\begin{figure}[t]
  \centering
  \includegraphics[width=.5\textwidth]{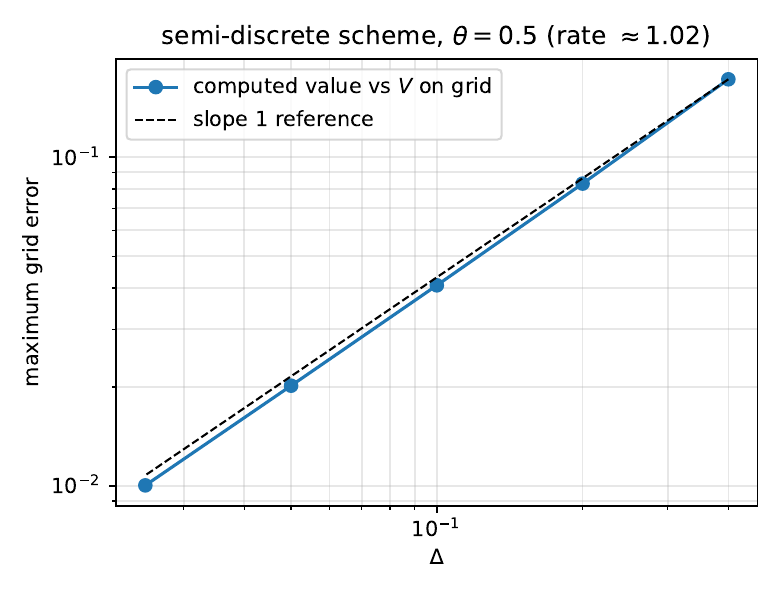}
  \caption{Maximum grid error of the implemented scheme versus \(\Delta\)
  (log--log), with a slope-one reference line.}
  \label{fig:scheme}
\end{figure}
\subsection{Statistical Protocol and Policy-Cost Comparisons}
\label{subsec:exp-stat}
\emph{Sampling and radius calibration.}
The true law is
\(\mu_*=\tfrac12N(-0.5,0.2^2)+\tfrac12N(0.7,0.3^2)\), its mean is
\(m_*=0.1\), and \(r=1\). We use the scalar dynamics \(f=u+w\),
\(U=[-1,1]\), \(g(x)=x^2/(1+x^2)\), \(\rho=1\), and \(x_0=1.5\).
The random seed is 7. For each \(N\in\{10,30,100,300,1000\}\),
1000 empirical laws are drawn, followed by 200 bootstrap resamples for each
draw. The oracle radius \(\theta_N\) is the empirical 90th percentile of
\(W_1(\mu_*,\nu_N)\) across the same replications used for evaluation.
The bootstrap radius \(\widehat\theta_N\) is the empirical 90th percentile
of \(W_1(\nu_N,\nu_N^b)\), evaluated by sorting the equally weighted atoms.
Both radii are simulation calibrations. Corollary~\ref{cor:statistical}
requires the separate concentration guarantee \eqref{eq:concentration}.

\emph{Wasserstein distance evaluation.}
Wasserstein distances to the true law are evaluated by integrating the mixture
cumulative distribution function (CDF) analytically between consecutive empirical atoms, splitting each
interval at its crossing with the empirical probability level. Crossings
are found by bracketed root finding with tolerance \(10^{-13}\), and the
two unbounded tails are included. These distances are used consistently
for oracle calibration and containment evaluation.
An independent numerical integration on \([-10,10]\), split at the atoms
and CDF crossings, in two selected replications at each sample size agrees
within \(2\times10^{-10}\); the discarded mixture tails contribute less
than \(10^{-200}\). A CDF grid approximation on \([-3,4]\) with spacing \(0.002\) differs
by up to \(5.51\times10^{-4}\) across all draws and is saved as a diagnostic.
The reported calibration and containment results use the CDF integrals.

\emph{Evaluation population.}
Coverage frequencies and bootstrap-radius summaries use all 1000 draws.
Value and policy-cost comparisons are restricted to
\[
 \max(\theta_N,\widehat\theta_N)+|\bar w_{\nu_N}|<0.97,
\]
so every one-sided recovery speed used in the value formula exceeds
\(0.03\). At \(N=10\), six draws fail the oracle condition and one
additional draw fails the bootstrap condition, leaving 993. All draws are
retained at the other sample sizes. Speeds are used without clipping.
All value and exceedance summaries below are conditional on this population.

\emph{Value calculation.}
For a mean \(\bar w\) and radius \(\theta\), define
\(k_R=1-\bar w-\theta\) and \(k_L=1+\bar w-\theta\).
On each half-line the value formula is \eqref{eq:exact-scalar}, with
\(\kappa\) replaced by the corresponding speed. Using distance from the
origin as the independent variable, we integrate \(kV'=g-\rho V\),
\(V(0)=0\), by fourth-order Runge--Kutta with step \(0.001\). Function errors are maximized over
\([-6,6]\cap0.02\mathbb Z\). All recorded errors on the retained draws
lie below the reference \(2L_c\theta_N/\rho^2\). For exact game values,
this bound is guaranteed on the containment event.

At \(x_0=1.5\), an independent 256-node Gauss--Legendre evaluation of
the discounted trajectory integral differs from the Runge--Kutta values
by at most \(4.1\times10^{-12}\) across the retained nominal, oracle,
and bootstrap cases. Refining that quadrature from 128 to 256 nodes changes
the result by at most \(1.12\times10^{-13}\).

\emph{Implemented policy and cost.}
At times \(t_k=k\delta\), \(\delta=0.001\), the control is set to
\(-1\) if \(X_{t_k}>0.01\), to \(1\) if \(X_{t_k}<-0.01\), and to
\(\operatorname{clip}(-\bar w_{\nu_N},-1,1)\) otherwise. It is held
constant on \([t_k,t_{k+1})\). The true relaxed trajectory obeys
\(\dot X=u+m_*\); its update over a holding interval is exact.
The discounted cost is integrated on each holding interval by four-node
Gauss--Legendre quadrature to time 30. The same policy is used for the
nominal, oracle, and bootstrap value comparisons in each replication.
Its suboptimality for these games has not been bounded.

The cost quadrature is checked on the identical state trajectories using
two nodes per interval. Across the retained replications, the largest
change is \(3.01\times10^{-13}\); the omitted discounted tail is at most
\(e^{-30}<9.36\times10^{-14}\). In exact arithmetic, the four-node
quadrature error has the bound
\[
 \frac{\delta^9(4!)^4}{9(8!)^3(1-e^{-\delta})}
  \sum_{j=0}^8 \binom{8}{j}(1+|m_*|)^j j!
 <1.3\times10^{-28}.
\]
Here \(|g^{(j)}|\le j!\) for \(j\ge1\), \(|g|\le1\), and
\(|\dot X|\le1+|m_*|\), so the sum bounds the eighth derivative of the
locally discounted integrand. This analytic estimate excludes rounding
error. Left-rectangle costs differ by up to \(3.47\times10^{-4}\) and are
saved as diagnostics.

\emph{Exceedance definition.}
We count \(J_{\rm num}>V_{\rm num}+\varepsilon_{\rm tol}\), with
\(\varepsilon_{\rm tol}=0.003\), and also report zero-tolerance results
to show sensitivity to the reporting threshold. This tolerance is not an
error bound for the policy or value approximation. The resulting rates
compare the specified policy cost with three value benchmarks; they do not
validate a feedback certificate.

\begin{table}[t]
\caption{Conditional cost exceedance percentages, at tolerances zero and
0.003; the retained populations are those in Section~\ref{subsec:stat-summary}}
\label{tab:consistency}
\centering\small
\begin{tabular}{rrrrrrr}
\toprule
&\multicolumn{2}{c}{Oracle}&\multicolumn{2}{c}{Nominal}
&\multicolumn{2}{c}{Bootstrap}\\
\(N\)&0&0.003&0&0.003&0&0.003\\
\midrule
10&3.8&3.3&51.0&47.7&7.9&7.2\\
30&3.4&3.0&48.6&44.3&5.5&4.4\\
100&3.2&2.6&46.5&39.1&3.3&2.8\\
300&4.0&1.4&51.5&38.8&5.4&1.8\\
1000&4.6&1.1&51.1&28.3&5.1&1.2\\
\bottomrule
\end{tabular}
\end{table}

The median bootstrap radii are \(0.324,0.199,0.113,0.065,0.036\).
Their containment frequencies over all replications are
81.5\%, 86.8\%, 89.9\%, 88.4\%, and 89.1\%, respectively, below the
nominal 90\% in these runs. Figure~\ref{fig:consistency} shows decreasing
value error with sample size and lower conditional exceedance against the
oracle benchmark than against the nominal value.

\begin{figure}[t]
\centering
\includegraphics[width=\textwidth]{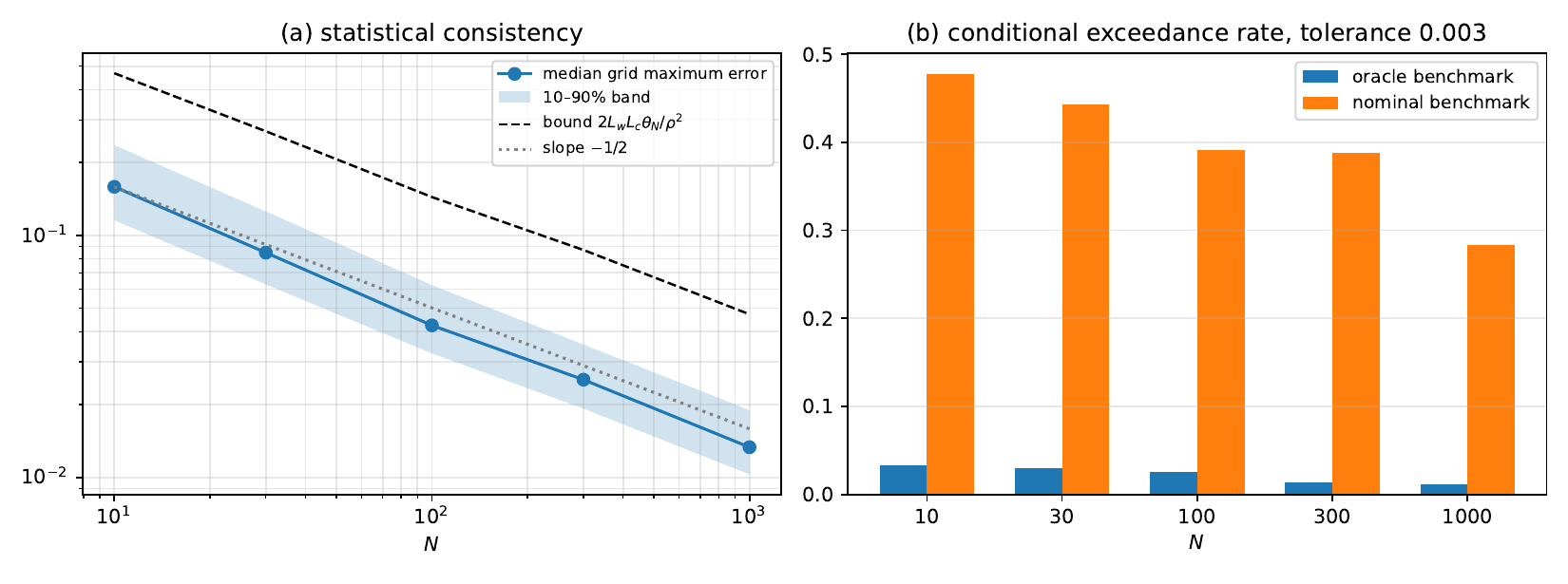}
\caption{Conditional grid value error and cost exceedance at tolerance
0.003; the oracle radius uses the same Monte Carlo replications for
calibration and evaluation}
\label{fig:consistency}
\end{figure}

\subsection{Inventory Policies under Increasing Ambiguity}
\label{subsec:exp-inventory}

This section specifies the model and numerical checks behind
Figure~\ref{fig:inventory}. Consider two products with inventory levels \(x\in\R^2\), production rates
\(u\in[0,1]^2\), demand \(w\in\R^2\), and dynamics \(\dot X=u-w\).
The running cost is
\(c(x,u)=\psi(x_1)+\psi(x_2)+0.25(u_1+u_2)\), where
\(\psi(t)=t_+/(1+t_+)+2t_-/(1+t_-)\), and \(\rho=1\).
Using seed 11, we draw 200 vectors from the Gaussian with mean
\((0.5,0.6)\) and covariance
\(\left(\begin{smallmatrix}0.02&0.008\\0.008&0.03\end{smallmatrix}\right)\),
and clip negative coordinates to zero. Their empirical mean is
\(\bar w_\nu=(0.50865308,0.60737428)\). The \(W_2\)-ball of radius
\(\theta\) around this empirical law acts through the exact mean ball
\(\bar B(\bar w_\nu,\theta)\), by Corollary~\ref{cor:mean-set}.

\emph{Discretization.}
We approximate \eqref{eq:TDelta} on \([-5,6]^2\) using bilinear
interpolation and constant extension at the computational boundary. The
reported array is denoted by \(v_{\Delta,h}\). For each control coordinate
we use \(\{0,0.25,0.5,0.75,1\}\), augmented by
\(\bar w_{\nu,i}\) and \(\min(1,\bar w_{\nu,i}+\theta/\sqrt2)\).
The product grid has 36 actions at \(\theta=0\) and 49 at the positive
radii tested. In particular, it contains the nominal equilibrium control.
The adversary uses the center and 32 equally spaced angles on each of the
circles of radii \(\theta/2\) and \(\theta\), giving 65 actions when
\(\theta>0\). Including both interior and boundary actions avoids assuming
that the interpolated value is maximized on the sphere.

The update retains the stage cost \(\Delta c\) and discount
\(e^{-\Delta}\) of \eqref{eq:TDelta}. A C++ kernel accelerates
Gauss--Seidel iteration by eliminating self-transition terms. The final
residual is evaluated for the original simultaneous Bellman update;
dividing it by \(1-e^{-\Delta}\) gives a finite-grid fixed-point error
bound below \(10^{-8}\) in every reported run, apart from floating-point
rounding. Independent NumPy and SciPy interpolation checks verify the
update. The residual measures the error in solving the finite-grid equation;
space, action, and domain effects are examined by refinement below.

\emph{Value accuracy.}
For \(\theta=0\) and \(x=0\), the admissible constant control
\(u=\bar w_\nu\) keeps the state at zero. Consequently,
\[
 V_0(0,0)\le J_{\rm hold}=0.25(\bar w_{\nu,1}+\bar w_{\nu,2})
 \approx0.27900684.
\]
The same policy has discrete cost
\(J_{{\rm hold},\Delta}=\Delta J_{\rm hold}/(1-e^{-\Delta})\).
The computed nominal origin values agree with this expression within
\(2\times10^{-8}\); its excess over the continuous-time policy cost is
\(0.0070333,0.0035021,0.0017474\) at the three time steps in
Table~\ref{tab:inventory-refinement}. This difference isolates the stage-cost
bias at a stationary state. The comparison uses a feasible-policy upper
bound for \(V_0\), without assuming that the holding policy is optimal.

\begin{table}[t]
\caption{Inventory origin values \(v_{\Delta,h}(0,0)\) with fixed action
sets and joint space--time refinement. These are computed grid values.}
\label{tab:inventory-refinement}
\centering\small
\begin{tabular}{rrrrrr}
\toprule
\(h=\Delta\)&\(\theta=0\)&0.1&0.2&0.3&0.4\\
\midrule
0.05&0.286040&0.327623&0.370569&0.411867&0.449786\\
0.025&0.282509&0.319939&0.358226&0.396554&0.435361\\
0.0125&0.280754&0.316853&0.353208&0.390576&0.427481\\
\bottomrule
\end{tabular}
\end{table}

The largest changes from \(h=\Delta=0.05\) and \(0.025\) to the
\(0.0125\) array, maximized over the five radii and common grid nodes in
\([-2,2.5]^2\), are \(0.0561\) and \(0.0187\), respectively.
At \(h=\Delta=0.0125\), augmenting a \(9\times9\) control grid by the
same balance actions and using 64 angles on four nonzero radii gives
100 controls and one adversarial action at \(\theta=0\), and
121 controls and 257 adversarial actions at \(\theta=0.4\).
The nominal value is unchanged; at \(\theta=0.4\) the origin value changes
from 0.427481 to 0.426218, and the maximum grid change is 0.002402.
The two reported production bands remain unchanged under this action
refinement. At \(h=\Delta=0.025\), enlarging the domain to
\([-10,12]^2\) changes values in the evaluation window by at most
0.000212 and 0.002122 at radii 0 and 0.4, respectively; the origin values
are unchanged to the displayed precision.
A separate check of the fine array at 529 states and 49 controls uses
128 angles on eight nonzero radii plus the center. The largest one-step
Bellman increase over the reported adversarial grid is
\(3.51\times10^{-5}\). Among the tested actions, interior points do not
improve on the refined boundary and center. These comparisons measure
sensitivity to the numerical choices; a global error estimate for the
continuous-game value remains unavailable.

\emph{Duality check.}
At 30 states selected with seed 12 from the evaluation window, with
\(p\) the centered numerical gradient of the fine-grid value at
\(\theta=0.4\), the closed-form Hamiltonian penalty and the dual minimized
over 4001 logarithmically spaced multipliers in \([10^{-3},10^3]\) agree
within \(6.9\times10^{-7}\) absolute and \(2.3\times10^{-6}\) relative
error. The minimizing multiplier agrees with
\(\lambda^*=|G^{\mathsf T}p|/(2\theta)\) to the multiplier-grid
resolution. A separate two-dimensional grid evaluation of the inner
supremum at five source atoms agrees with its analytic form within
\(2.6\times10^{-5}\). These tests concern the Hamiltonian duality formula.

\emph{Sensitivity and policy.}
On the fine evaluation grid in \([-2,2.5]^2\), the largest value
increases from radius zero are \(0.1114,0.2316,0.3621,0.5002\) for
\(\theta=0.1,0.2,0.3,0.4\). They are approximately linear and below
the corresponding theoretical reference bounds in Theorem~\ref{thm:sensitivity}.
Figure~\ref{fig:inventory} uses \(h=\Delta=0.0125\).
Along \(x_2=0.5\) across the full computational domain, the computed
full-production band changes from
\([-2.0375,-0.0125]\) at radius zero to
\([-1.4125,-0.0125]\) at radius 0.4; the full-production area fraction
changes from 0.4432 to 0.3329 on the evaluation grid, including its boundary nodes.
At \((0,0.5)\), production of the first product increases from 0.5087
to 0.7915. The upper endpoint of the full-production band stays fixed
at this resolution, while the lower endpoint moves toward balance.
Production therefore increases near balance and retreats from full
production in deeper backlogs, where the saturating cost has a smaller
marginal response to replenishment.

\subsection{Nonlinear Model: Numerical Procedure}
\label{subsec:nonaffine-procedure}
For the \(\tanh\) model in Section~\ref{subsec:exp-nonaffine}, the dual
reduces to scalar maximizations and a one-dimensional multiplier search.
For \(\lambda>0\), each scalar maximizer lies in
\([\widehat w_i,\widehat w_i+1/(2\lambda)]\).
We split this interval at the zeros of the second derivative, locate every
stationary point by bracketed root finding, and compare the candidate values.
For the lower endpoint the same procedure is applied after reflecting the
atoms. The convex scalar dual is minimized over the compact interval
\([0,(1-\overline{\phi})/\theta^2]\), where
\(\phi=\tanh\) or \(-\tanh\) and
\(\overline{\phi}=\frac1{200}\sum_i\phi(\widehat w_i)\).
This interval follows from
\(D(\lambda)\ge\lambda\theta^2+\overline{\phi}\) and \(D(0)=1\).
Root and multiplier tolerances are \(10^{-13}\) and \(10^{-11}\).
Stationary-point enumeration searches the continuous inner problem.
The value curves in Figure~\ref{fig:nonaffine} are computed by
fourth-order Runge--Kutta with spatial step \(0.0005\).

A projected particle calculation supplies feasible distributions, with
separately recorded signed gaps below \(10^{-10}\) in absolute value.
This comparison uses one transported particle per source atom. For a
general integrand, splitting an empirical atom can be necessary
\citep[Cor.~2]{GaoKleywegt}, so the particle calculation alone does not
establish primal optimality.